\documentclass[reqno, 11pt, letterpaper]{amsart}
\usepackage{amsfonts,amsmath, amsthm, amssymb, latexsym, epsfig}
\usepackage[all]{xy}
\usepackage{xspace}
\usepackage{comment}
\usepackage{setspace}
\usepackage{enumerate}
\usepackage[mathscr]{eucal}

\usepackage{bbold}
\newcommand{\fifi}{\pmb{[11]}}
\newcommand{\fiin}{\pmb{[1\infty]}}
\newcommand{\inin}{\pmb{[\infty\infty]}}
\newcommand{\infi}{\pmb{[\infty 1]}}

	\advance \topmargin by -\headheight
	\advance \topmargin by -\headsep

	\evensidemargin \oddsidemargin
	\newcommand{\ftn}[3]{ #1 : #2 \rightarrow #3 }
		\newcommand{\setof}[2]{\ensuremath{\left\{ #1 \: : \: #2 \right\}}}

	\newcommand{\ZZ}{\ensuremath{\mathbb{Z}}\xspace}
	\newcommand{\CC}{\ensuremath{\mathbb{C}}\xspace}

	\newcommand{\RR}{\ensuremath{\mathbb{R}}\xspace}
	
	\newcommand{\KKK}{\ensuremath{\mathbb{K}}\xspace}

	\newcommand{\kk}{\ensuremath{\mathit{KK}}\xspace}

	\newcommand{\id}{\ensuremath{\operatorname{id}}}

	\newcommand{\multialg}[1]{\mathcal{M}(#1)\xspace}
	\newcommand{\corona}[1]{\mathcal{Q}(#1)\xspace}

	\newcommand{\cstar}{{$C \sp \ast$}\xspace}

		\newcommand{\Z}{\ensuremath{\mathbb{Z}}\xspace}
	\newcommand{\C}{\ensuremath{\mathbb{C}}\xspace}
	\newcommand{\Q}{\ensuremath{\mathbb{Q}}\xspace}
	
	\newcommand{\N}{\ensuremath{\mathbb{N}}\xspace}
	
	\newcommand{\K}{\ensuremath{\mathbb{K}}\xspace}
	\newcommand{\ksix}{\ensuremath{K_{\mathrm{six}}}\xspace}

	\theoremstyle{plain}
	\newtheorem{thm}{Theorem}[section]
	\newtheorem{lemma}[thm]{Lemma}
	\newtheorem{theor}[thm]{Theorem}
	\newtheorem{propo}[thm]{Proposition}
	\newtheorem{corol}[thm]{Corollary}

	\theoremstyle{definition}
	\newtheorem{defin}[thm]{Definition}
	
	\newtheorem{remar}[thm]{Remark}
	\newtheorem{notat}[thm]{Notation}
	\newtheorem{examp}[thm]{Example}

	\numberwithin{equation}{section}
	\numberwithin{figure}{section}

\begin{document}
	\title{The ordered $K$-theory of a full extension}
	\author{S{\o}ren Eilers}
        \address{Department of Mathematical Sciences \\
        University of Copenhagen\\
        Universitetsparken~5 \\
        DK-2100 Copenhagen, Denmark}
        \email{eilers@math.ku.dk }
        \author{Gunnar Restorff}
\address{Faculty of Science and Technology\\University of Faroe 
Islands\\N\'oat\'un 3\\FO-100 T\'orshavn\\Faroe Islands}
\email{gunnarr@setur.fo}
	\author{Efren Ruiz}
        \address{Department of Mathematics\\University of Hawaii,
Hilo\\200 W. Kawili St.\\
Hilo, Hawaii\\
96720-4091 USA}
        \email{ruize@hawaii.edu}
        \date{\today}
\thanks{This research was supported by the Danish National Research Foundation (DNRF) through the
Centre for Symmetry and Deformation.  Support was also provided by the NordForsk Research Network
  ``Operator Algebra and Dynamics'' (grant \#11580), and by the Faroese Research Council. }
	


	\begin{abstract}
	Let $\mathfrak{A}$ be a $C^{*}$-algebra with real rank zero which has the stable weak cancellation property.  Let $\mathfrak{I}$ be an ideal of $\mathfrak{A}$ such that $\mathfrak{I}$ is stable and satisfies the corona factorization property.  We prove that 	
	\begin{align*}
	0 \to \mathfrak{I} \to \mathfrak{A} \to \mathfrak{A} / \mathfrak{I} \to 0
	\end{align*}
	is a full extension if and only if the extension is stenotic and $K$-lexicographic.  {As an immediate application, we extend the classification result for graph $C^*$-algebras obtained by Tomforde and the first named author to the general non-unital case. In combination with recent results  by Katsura, Tomforde, West and the first named author, our result may also be used to give a purely $K$-theoretical description of when an essential extension of two simple and stable graph $C^*$-algebras is again a graph $C^*$-algebra.}
	\end{abstract}

        \maketitle

\section{Introduction}        
        
An extension 
\begin{align*}
\mathfrak{e} : 0 \to \mathfrak{A} \to \mathfrak{E} \to \mathfrak{B} \to 0
\end{align*}
of $C^{*}$-algebras is called \emph{full} if the image of any nonzero element of $\mathfrak{B}$ under
the Busby map defining $\mathfrak{e}$ is never contained in a proper ideal of the corona
algebra $\corona{\mathfrak{A}} = \multialg{\mathfrak{A}} / \mathfrak{A}$.  Note that if the extension is essential and $\corona{\mathfrak{A}}$ is simple (as in the case when $\mathfrak{A} = \K$), then $\mathfrak{e}$ is a full extension.  Hence, the Brown-Douglas-Fillmore Theory (\cite{BDF1} and \cite{BDF2}) can be considered as the study of full extensions of $C(X)$ by $\K$.  In \cite{PPV1} and \cite{PPV2}, Pimsner, Popa, and Voiculescu studied homogeneous extensions of $C(X)\otimes \K$ by $\mathfrak{A}$ which turns out to be full extensions.  Motivating the study of essential extensions of $C(X)$ by $\mathfrak{A}$, when $\corona{\mathfrak{A}}$ is simple (\cite{HLextI} and \cite{HLextII}).  These extensions were later characterized by Elliott and Kucerovsky (\cite{gaedk:avbfat}) which they called ``purely large extensions'' and this notion has turned out to be of great importance for the classification theory of
non-simple infinite $C^{*}$-algebras.

In fact, the notion appeared in disguise already in the first such
classification result by R{\o}rdam \cite{extpurelyinf}, when both $\mathfrak{A}$ and $\mathfrak{B}$ are stable and purely
infinite, but did not need explicit mention since both $\mathfrak{B}$ and $\corona{ \mathfrak{A} }$ are
simple in that case. In a series of papers (see \cite{ERRshift}, \cite{ERRlinear}, and \cite{segrer:ccfis}), the authors have used this notion
to prove classification results for a host of $C^{*}$-algebras related to graphs
and dynamical systems, where fullness is automatic. Advances have also been made in other cases, but progress here has been slow  due to lack
of this property.

{In the present paper, in the case when the real rank of $\mathfrak E$ is zero, and under certain technical assumptions on $\mathfrak A$ and $\mathfrak B$ which hold automatically in the setting we are interested in, we give a  characterization of when $\mathfrak e$ is full.  Using (slightly modified) terminologies introduced by Goodearl and Handelman in \cite{kgdh:stenosis} and Handelman in \cite{dh:extAF}, we show that $\mathfrak{e}$ is full if and only if $\mathfrak{e}$ is stenotic and $K$-lexicographic.  Our main result of this kind is presented in Theorem~\ref{t:full} and may be slightly improved under further assumptions. When the ideal lattice in question is linear (cf. \cite{ERRlinear}), a very efficient reformulation is given in Corollary~\ref{lincase}. In most of our results  we rely on the notion of \textsl{weak cancellation} investigated by 
Ara, Moreno, and Pardo in \cite{amp:nonstablekthy}. Although possibly true for all real rank zero $C^*$-algebras, the range of this property is not fully understood, but it is known to hold in the case of graph $C^*$-algebras which has been our main concern in applications.}

{We draw attention to four different applications of this characterization. \textbf{First} -- our original motivation -- it explains why, in most of the classification results presented in \cite{ERRshift}, \cite{ERRlinear}, and \cite{semt_classgraphalg}, the order of $K_0(\mathfrak E)$ may be deleted from the invariant leaving it still complete. The explanation for this phenomenon is simply that we explicitly or implicitly call for fullness in the theorems, and that fullness in turn forces the order of $K_0(\mathfrak E)$ to be of a certain form which may be recovered from the order of $K_0(\mathfrak A)$ and $K_0(\mathfrak B)$. We can also, \textbf{second}, give an example showing the fact that even if one is prepared to use the order on $K_0(\mathfrak E)$, the invariants are not complete for non-full essential extensions of classifiable simple $C^*$-algebras.

\textbf{Third}, our concrete $K$-theoretical test
for fullness combines with results in a paper by Katsura,
Tomforde, West and the first named author (\cite{ektw}) in which the range of the invariants for one-ideal
graph $C^{*}$-algebras is determined, and with the classification results in \cite{semt_classgraphalg}, to lead
to the first permanence result for extensions of graph $C^{*}$-algebras.
 \textbf{Fourth}, it allows the completion of a result from \cite{ERRshift} to establish that
extensions with a stable ideal are full precisely when their stabilizations
are, leading us in turn to an exact classification result in some cases solved
up to stable isomorphism in \cite{semt_classgraphalg}. 

\section{Preliminaries} 
 Throughout the entire paper $\{ e_{ij} \}_{ i , j = 1}^{ \infty }$ will denote a system of matrix units for $\K$.  For a $C^{*}$-algebra $\mathfrak{A}$, $\multialg{ \mathfrak{A} }$ will denote the multiplier algebra of $\mathfrak{A}$ and $\corona{\mathfrak{A}} $ will denote the corona algebra $\multialg{ \mathfrak{A} } / \mathfrak{A}$ of $\mathfrak{A}$. We say an element $a \in \mathfrak{A}$ is \emph{norm-full in $\mathfrak{A}$} if $a$ is not contained in any normed closed proper ideal of $\mathfrak{A}$.  We say that a sub-$C^{*}$-algebra $\mathfrak{B}$ of $\mathfrak{A}$ is \emph{norm-full in $\mathfrak{A}$} if $\mathfrak{B}$ is not contained in any normed closed proper ideal of $\mathfrak{A}$.  The term ``full'' is also widely used, but since we will often work in multiplier algebras, we emphasize that it is the norm topology we are using, rather than the strict topology.

 \begin{defin}
Let $\mathfrak{I}$ be an ideal of the $C^*$-algebra $\mathfrak{A}$. Let $\ftn{ \theta_{ \mathfrak{I} } }{ \mathfrak{I} }{ \multialg{ \mathfrak{I} } }$ be the canonical embedding of $\mathfrak{I}$ as an essential ideal of $\multialg{ \mathfrak{I} }$.  Let $\ftn{ \sigma_{ \mathfrak{e} } }{ \mathfrak{A} }{ \multialg{ \mathfrak{I} } }$ be the homomorphism such that the diagram
\begin{equation*}
\xymatrix{
0 \ar[r] & \mathfrak{I} \ar[r] \ar@{=}[d]& \mathfrak{A} \ar[r] \ar[d]^{\sigma_{ \mathfrak{e} } } & \mathfrak{A} / \mathfrak{I} \ar[r] \ar[d]^{ \tau_{ \mathfrak{e} } }& 0 \\
0 \ar[r] & \mathfrak{I} \ar[r]_-{ \theta_{ \mathfrak{I} } }  & \multialg{ \mathfrak{I} } \ar[r]_-{ \pi_{ \mathfrak{I} } }  & \corona{  \mathfrak{I} } \ar[r] & 0 
}
\end{equation*}
is commutative, where $\tau_{ \mathfrak{e} }$ is the Busby invariant of the extension $\mathfrak{e} : 0 \to \mathfrak{I} \to \mathfrak{A} \to \mathfrak{A} / \mathfrak{I} \to 0.$
\end{defin}

\begin{notat}
Let $\mathfrak{A}$ be a $C^{*}$-algebra and $\mathfrak{I}, \mathfrak{D}$ be ideals of $\mathfrak{A}$ such that $\mathfrak{I} \subseteq \mathfrak{D}$.  
\begin{itemize}
\item[(1)] Let $\ftn{ \iota_{ \mathfrak{I} , \mathfrak{D} } }{ \mathfrak{I} }{ \mathfrak{D} }$ be the canonical inclusion map and $\ftn{ \pi_{\mathfrak{I} , \mathfrak{D} } }{ \mathfrak{D} }{ \mathfrak{D} / \mathfrak{I} }$ be the canonical projection. 

\item[(2)] If $\ftn{ \phi }{ \mathfrak{A} }{ \mathfrak{B} }$ is a homomorphism between $C^{*}$-algebras, then the homomorphism $\ftn{ \phi \otimes \id_{ \K } }{ \mathfrak{A} \otimes \K }{ \mathfrak{B} \otimes \K }$ will be denoted by $\phi^{s}$.
\end{itemize}
\end{notat}
 
 \begin{defin}
Let $\mathfrak{A}$ be a $C^{*}$-algebra and let $\mathfrak{B}$ be a sub-$C^{*}$-algebra of $\mathfrak{A}$.  For projections $p, q \in \mathfrak{B}$, we write \emph{$p \sim q$ in $\mathfrak{B}$} if there exists $v \in \mathfrak{B}$ such that $v^{*} v = p$ and $v v^{*} = q$ and we write \emph{$p \lesssim q$ in $\mathfrak{B}$}, if there exists a projection $e$ in $\mathfrak{B}$ such that $p \sim e$ in $\mathfrak{B}$ and $e \leq q$.  If $p \sim q$ in $\mathfrak{A}$ ($p \lesssim q$ in $\mathfrak{A}$), we write $p \sim q$ (respectively, $p \lesssim q$).
\end{defin}
 
We end this section with a series of lemmas that will be used throughout the next section. 

\begin{lemma}\label{l:her}
Let $\mathfrak{A}$ be a separable $C^{*}$-algebra and let $\mathfrak{B}$ be a hereditary sub-$C^{*}$-algebra of $\mathfrak{A}$.  Suppose $p$ and $q$ are projections in $\mathfrak{B}$ such that $p \sim q$.  Then $p \sim q$ in $\mathfrak{B}$.  Consequently, if $p$ and $q$ are projections in $\mathfrak{B}$ and $p \lesssim q$, then $p \lesssim q$ in $\mathfrak{B}$.
\end{lemma} 

\begin{proof}
Suppose $p \sim q$.  Then there exists $v \in \mathfrak{A}$ such that $v^{*} v = p$ and $v v^{*}  = q$.  Set $w = q v p$.  Then 
\begin{equation*}
w^{*} w = p v^{*} q v p = p \quad \text{and} \quad w w^{*} = q v p v^{*} q = q.
\end{equation*}
Since $\mathfrak{A}$ is separable, $\mathfrak{B} = \overline{ a \mathfrak{A} a }$ for some $a \in \mathfrak{A}_{+}$.  Hence, $w \in \overline{ a \mathfrak{A} a } = \mathfrak{B}$.

Suppose $p \lesssim q$.  Then there exists a projection $e \in \mathfrak{A}$ such that $p \sim e$ and $e \leq q$.  Since $q \in \mathfrak{B}$ and $\mathfrak{B}$ is a hereditary sub-$C^{*}$-algebra of $\mathfrak{A}$, $e \in \mathfrak{B}$.  Hence, $p \sim e$ in $\mathfrak{B}$ and $e \leq q$.  Therefore, $p \lesssim q$ in $\mathfrak{B}$.
\end{proof}
 
Adapting the proof of Lemma~2.6 of \cite{SZ_proj}, we get the following lemma.
 
\begin{lemma}\label{l:strict}
Let $\mathfrak{A}$ be a $C^{*}$-algebra and let $\mathfrak{I}$ be a non-unital ideal of $\mathfrak{A}$.  Let $p, q \in \mathfrak{A}$ be projections such that there exists a projection $e \in \mathfrak{A} / \mathfrak{I}$ with $\pi_{ \mathfrak{I} , \mathfrak{A} } ( p ) \sim e$ and $e \leq \pi_{ \mathfrak{I}, \mathfrak{A} } ( q )$.  Suppose $p \mathfrak{I}  p$ has an approximate identity consisting of projections.  Then there exist projections $p_{1} \in p \mathfrak{A} p$, $e_{1} \in q \mathfrak{A} q$ such that $p_{1} \sim e_{1}$, $\pi_{ \mathfrak{I} , \mathfrak{A} } ( p_{1} ) = \pi_{ \mathfrak{I} , \mathfrak{A} } ( p )$, and $\pi_{ \mathfrak{I} , \mathfrak{A} } ( e_{1}  ) < \pi_{ \mathfrak{I} , \mathfrak{A} } ( q )$ if $e < \pi_{ \mathfrak{I}, \mathfrak{A} } ( q )$ and $\pi_{ \mathfrak{I} , \mathfrak{A} } ( e_{1}  ) = \pi_{ \mathfrak{I} , \mathfrak{A} } ( q )$ if $e = \pi_{ \mathfrak{I}, \mathfrak{A} } ( q )$.
\end{lemma}

\begin{proof}
Let $v \in \mathfrak{A}/\mathfrak{I}$ such that $v^{*} v = \pi_{ \mathfrak{I} , \mathfrak{A} } ( p )$ and $v v^{*} \leq \pi_{ \mathfrak{I} , \mathfrak{A} } ( q )$, with strict inequality if $e < \pi_{ \mathfrak{I}, \mathfrak{A} } ( q )$.  Take $V \in \mathfrak{A}$ such that $\pi_{ \mathfrak{I} , \mathfrak{A} } ( V ) = v$.  Set $w = q V p$.  Then
\begin{equation*}
\pi_{ \mathfrak{I} , \mathfrak{A} } ( w^{*} w ) = \pi_{ \mathfrak{I} , \mathfrak{A} } ( p ) 
\end{equation*}  
and 
\begin{equation*}
\pi_{ \mathfrak{I} , \mathfrak{A} } ( w w^{*} ) = \pi_{ \mathfrak{I} , \mathfrak{A} }( q ) v v^{*} \pi_{ \mathfrak{I} , \mathfrak{A} } ( q ) = v v^{*} \leq\pi_{ \mathfrak{I} , \mathfrak{A} } ( q )
\end{equation*}
with strict inequality if $e < \pi_{ \mathfrak{I}, \mathfrak{A} } ( q )$.

Since $w^{*} w - p \in p \mathfrak{I}  p$ and since $p \mathfrak{I} p$ has an approximate identity consisting of projections, there exists a projection $f \in p \mathfrak{I}  p$ such that 
\begin{equation*}
\| (p - f ) ( w^{*} w - p )( p - f ) \| < 1.
\end{equation*}
Hence, there exists $g \in ( p - f ) \mathfrak{A} (p - f )$ positive such that $g ( p - f ) w^{*} w (p-f) = ( p - f ) w^{*} w ( p - f ) g = p - f$.

Set $z =  wg^{ \frac{1}{2} } $.  Then $z^{*} z = g^{\frac{1}{2} }  w^{*} w g^{ \frac{1}{2} } = g^{ \frac{1}{2} } ( p - f ) w^{*} w (p-f) g^{ \frac{1}{2} } = p - f \in p \mathfrak{A} p$ and $z z^{* } = w g w^{*} =  q V p g p V^{*} q \in q \mathfrak{A} q$.  Set $p_{1} = z^{*} z$ and $e_{1} = z z^{*}$.  Then
\begin{equation*}
\pi_{ \mathfrak{I} , \mathfrak{A} } ( p_{1} ) = \pi_{ \mathfrak{I} , \mathfrak{A} } ( p - f ) = \pi_{ \mathfrak{I} , \mathfrak{A} } ( p ).
\end{equation*}
Note that $\pi_{ \mathfrak{I} , \mathfrak{A} } ( g ) \pi_{ \mathfrak{I} , \mathfrak{A} } ( p ) = \pi_{ \mathfrak{I} , \mathfrak{A} } ( p )$.  Therefore, 
\begin{align*}
\pi_{ \mathfrak{I} , \mathfrak{A} } ( z z^{*} ) = \pi_{ \mathfrak{I} , \mathfrak{A} } ( w ) \pi_{ \mathfrak{I} , \mathfrak{A} } ( p ) \pi_{ \mathfrak{I} , \mathfrak{A} } ( w^{*} ) = \pi_{ \mathfrak{I} , \mathfrak{A} } ( w )  \pi_{ \mathfrak{I} , \mathfrak{A} } ( w^{*} w ) \pi_{ \mathfrak{I} , \mathfrak{A} } ( w^{*} ) = \pi_{ \mathfrak{I} , \mathfrak{A} } ( ww^{*} ) \leq \pi_{ \mathfrak{I} ,\mathfrak{A} } ( q )
\end{align*}
with strict inequality if $e < \pi_{ \mathfrak{I}, \mathfrak{A} } ( q )$.
\end{proof}

\begin{remar}
By \cite[Theorem~6.5.6]{blackadarB}, hereditary subalgebras of real rank zero $C^{*}$-algebras have an approximate identity consisting of projections.  This fact will be used throughout the paper.  We will also use the facts that real rank zero passes through ideals, quotients, and stabilization (see \cite[Theorem~3.14 and Corollary~3.3]{realrank}).
\end{remar}

The following observation follows easily from the fact that $\sigma_{ \mathfrak{e}}$ is injective.

\begin{lemma}\label{l:essentialmv}
Let $\mathfrak{A}$ be a $C^{*}$-algebra and let $\mathfrak{B}$ be an essential ideal of $\mathfrak{A}$.  If $p$ and $q$ are projections in $\mathfrak{A}$ such that $\sigma_{ \mathfrak{e} } ( p ) \leq \sigma_{ \mathfrak{e} } ( q )$, then $p \leq q$.
\end{lemma}

\begin{lemma}\label{lem:projfull}
 Let $\mathfrak{A}$ be a non-zero separable $C^{*}$-algebra with real rank zero.  Suppose $p \in \multialg{ \mathfrak{A} \otimes \K }$ is a projection such that for all projections $e, f \in \mathfrak{A} \otimes \K$ with $f \leq p$ we have that $e  \lesssim p - f$ in $\multialg{ \mathfrak{A} \otimes \K }$ and $p ( \mathfrak{A} \otimes \K ) p$ is norm-full in $\mathfrak{A} \otimes \K$.  Then $1_{ \multialg{ \mathfrak{A} \otimes \K } } \sim p$.   
 \end{lemma}
 
 \begin{proof}
First note that the assumption on $p$ implies that $p$ is not an element of $\mathfrak{A} \otimes \K$.  We claim that $p ( \mathfrak{A} \otimes \K ) p$ is a stable $C^{*}$-algebra.  Let $e$ be a nonzero projection in $p ( \mathfrak{A} \otimes \K ) p$.  By our assumption, $e \lesssim p - e$.  Therefore, there exists a nonzero projection $q \in (p - e ) (\mathfrak{A} \otimes \K ) ( p - e )$ such that $e \sim q$ in $p ( \mathfrak{A} \otimes \K ) p$.  Hence, by Theorem~3.3 of \cite{HjelmRord}, $p (\mathfrak{A} \otimes \K ) p$ is a stable $C^{*}$-algebra.  Thus, by Theorem~4.23 of \cite{BrownSemi}, $p \sim 1_{ \multialg{ \mathfrak{A} \otimes \K } }$.
\end{proof}

The following lemma is well-known but we were not able to find a reference for it so we provide the proof here.

\begin{lemma}\label{l:kthyrr0}
Let $\mathfrak{A}$ be a $C^{*}$-algebra with real rank zero.  If $x \in K_{0} ( \mathfrak{A} )$, then there exist projections $p$ and $q$ in $\mathsf{M}_{n} ( \mathfrak{A} )$ such that $x = [ p ] - [ q ]$.
\end{lemma}

\begin{proof}
If $\mathfrak{A}$ is unital, then this is clear.  Suppose $\mathfrak{A}$ is non-unital.  Denote the unitization of $\mathfrak{A}$ by $\widetilde{\mathfrak{A}}$ and let $\ftn{ \pi }{ \widetilde{ \mathfrak{A} } }{ \C }$ be the natural projection.  Let $x \in K_{0} ( \mathfrak{A} )$.  Then $x = [ p_{1} ] - [ q_{1} ]$ where $p_{1}, q_{1}$ are projections in $\mathsf{M}_{n} ( \widetilde{\mathfrak{A}} )$ such that $K_{0} ( \pi ) ( [p_{1}] ) = K_{0} ( \pi ) ( [ q_{1} ] )$.  Since $\C$ has stable rank one, then $\pi ( p_{1} ) \sim \pi ( q_{1} )$.

By Lemma~\ref{l:strict}, there exists $v \in \mathsf{M}_{n} ( \widetilde{ \mathfrak{A} } )$ such that $v^{*} v \leq p_{1}$, $\pi ( v^{*} v ) = \pi ( p_{1} )$, $vv^{*} \leq q_{1}$, and $\pi ( vv^{*} ) = \pi ( q_{1} )$.  Therefore, $p_{1} - v^{*} v$ and $q_{1} - vv^{*}$ are projections in $\mathsf{M}_{n} ( \mathfrak{A} )$ and 
\begin{align*}
x = [ p_{1} ] - [ q_{1} ] = [ p_{1} - v^{*} v ] + [ v^{*} v ] - [ q_{1} ] = [ p_{1} - v^{*} v ] - [ q_{1} - v v^{*} ]. 
\end{align*}
Set $p = p_{1} - v^{*} v$ and $q= q_{1} - vv^{*}$.
\end{proof}
        
 \section{Full extensions and ordered $K$-theory}
 
In this section, we show, for a certain class of extensions which includes extensions arising from graphs, that the ordered $K$-theory of a stenotic extension with real rank zero has the ``lexicographic ordering'' precisely when the extension is full.  The results of this section explains why, in most of the classification results presented in \cite{ERRshift}, \cite{ERRlinear}, and \cite{semt_classgraphalg}, the order of $K_0(\mathfrak E)$ may be deleted from the invariant leaving it still complete.  Moreover, in the next section, we show that our concrete $K$-theoretical test
for fullness combines with results in a paper by Katsura,
Tomforde, West and the first named author (\cite{ektw}) in which the range of the invariants for one-ideal
graph $C^{*}$-algebras is determined, and with the classification results in \cite{semt_classgraphalg}, to lead
to the first permanence result for extensions of graph $C^{*}$-algebras.  

\begin{defin}
An extension $\mathfrak{e} : 0 \to \mathfrak{I} \to \mathfrak{A} \to \mathfrak{A}/\mathfrak{I} \to 0$ is \emph{full} if $\tau_{ \mathfrak{e} } ( a )$ is norm-full in $\corona{ \mathfrak{I}}$ for all nonzero $a \in \mathfrak{A}/\mathfrak {I}$.
\end{defin}

The following definition is taken from \cite{kgdh:stenosis}.
\begin{defin}
An extension $\mathfrak{e} : 0 \to \mathfrak{I} \overset{ \iota }{ \to  } \mathfrak{A} \overset{ \pi }{ \to } \mathfrak{B} \to 0$
is called \emph{stenotic} if $\mathfrak{D} \subseteq \iota ( \mathfrak{I} )$ or $\iota ( \mathfrak{I} ) \subseteq \mathfrak{D}$ for all ideals $\mathfrak{D}$ of $\mathfrak{A}$.
\end{defin}

Note that if $0 \to \mathfrak{I} \to \mathfrak{A} \to \mathfrak{A} / \mathfrak{I} \to 0$ is stenotic, where $\mathfrak{I}$ is a nonzero ideal of $\mathfrak{A}$, then $\mathfrak{I}$ is an essential ideal of $\mathfrak{A}$.  Also, an extension $\mathfrak{e}: 0 \to \mathfrak{I} \to \mathfrak{A} \to \mathfrak{B} \to 0$
is stenotic if and only if the extension $\mathfrak{e}^{s} : 0 \to \mathfrak{I} \otimes \K \to \mathfrak{A} \otimes \K \to \mathfrak{B} \otimes \K \to 0$ is stenotic.

We will prove, under suitable hypothesis which are satisfied by the $C^{*}$-algebras we study in this paper, that every full extension is stenotic.  To do this we need the following lemma, which is similar to that of Lemma~3.3 of \cite{HL_fullext}.  We thank Ping Wong Ng for helpful discussions.

\begin{lemma}\label{l:fullcorona}
Let $\mathfrak{A}$ be a $C^{*}$-algebra with an approximate identity consisting of projections.  If $x \in \multialg{ \mathfrak{A} \otimes \K }$ such that $\pi_{ \mathfrak{A} \otimes \K } ( x )$ is norm-full in $\corona{ \mathfrak{A} \otimes \K }$, then $x$ is norm-full in $\multialg{ \mathfrak{A} \otimes \K }$.
\end{lemma}

\begin{proof}
Let $e_{n} = 1_{ \multialg{ \mathfrak{A} } }\otimes \left( 1 _{ \multialg{ \K } } - \sum_{ k = 1}^{n} e_{kk} \right) \in \multialg{ \mathfrak{A} \otimes \K }$.  Note that $\{ e_{n} \}_{ n = 1}^{ \infty }$ is a sequence of projections in $\multialg{ \mathfrak{A} \otimes \K }$ which converges to $0$ in the strict topology.  Moreover, $e_{n} \sim 1_{ \multialg{ \mathfrak{A} \otimes \K } }$ since $1 _{ \multialg{ \K } } - \sum_{ k = 1}^{n} e_{kk} \sim 1_{ \multialg{ \K } }$.  Thus, there exists $V_{n} \in \multialg{ \mathfrak{A} \otimes \K }$ such that $V_{n}^{*} V_{n} = 1_{ \multialg{ \mathfrak{A} \otimes \K } }$ and $V_{n} V_{n}^{*} = e_{n}$. 

Suppose $x \in \multialg{ \mathfrak{A} \otimes \K }$ such that $\pi_{ \mathfrak{A} \otimes \K } (x)$ is norm-full in $\corona{ \mathfrak{A} \otimes \K }$.  Then there exist $a_{i}, b_{i} \in \multialg{ \mathfrak{A} \otimes \K }$ such that 
\begin{align*}
\left\| \pi_{ \mathfrak{A} \otimes \K } ( 1_{ \multialg{ \mathfrak{A} \otimes \K } }  ) - \sum_{ i = 1}^{m} \pi_{ \mathfrak{A} \otimes \K } ( a_{i} x b_{i} ) \right\| < \frac{1}{2}.
\end{align*}
Thus,
\begin{align*}
\left\|1_{ \multialg{ \mathfrak{A} \otimes \K } }  - \sum_{ i = 1}^{m} a_{i} x b_{i} - d \right\| < \frac{1}{2} 
\end{align*}
for some $d \in \mathfrak{A} \otimes \K$.  Since $\mathfrak{A} \otimes \K$ has an approximate identity consisting of projections, there exists a projection $e \in \mathfrak{A} \otimes \K$ such that $\left\| ( 1_{ \multialg{ \mathfrak{A} \otimes \K } } - e ) d ( 1_{ \multialg{ \mathfrak{A} \otimes \K }} - e ) \right\| < \frac{1}{2}$.  Set $p = 1_{ \multialg{ \mathfrak{A} \otimes \K } } - e$.  Hence,
\begin{align*}
\left\| p - \sum_{ i = 1}^{m} p a_{i} x b_{i} p \right\|  &\leq \left\| p - \sum_{ i = 1}^{m} p a_{i} x b_{i} p - p d p \right\| + \left\| p d p \right\|. \\
&< 1.
\end{align*}
Thus, $\sum_{ i = 1}^{m} p a_{i} x b_{i} p$ is an invertible element in $p \multialg{ \mathfrak{A} \otimes \K } p$, i.e., there exists $y \in p \multialg{ \mathfrak{A} \otimes \K } p$ such that 
\begin{align*}
y\left(\sum_{ i = 1}^{m} pa_{i} x b_{i} p\right) = p.
\end{align*}

Since $\{ e_{n} \}_{ n= 1}^{ \infty }$ converges to $0$ in the strict topology and $e \in \mathfrak{A} \otimes \K$, 
\begin{align*}
\lim_{n \to \infty }\| 1_{ \multialg{ \mathfrak{A} \otimes \K } } - V_{n}^{*} p V_{n} \| &= \lim_{n \to \infty }\| 1_{ \multialg{ \mathfrak{A} \otimes \K } } - V_{n}^{*} ( 1_{ \multialg{ \mathfrak{A} \otimes \K } } - e ) V_{n} \| \\
		&= \lim_{ n \to \infty }\| V_{n}^{*} e V_{n} \|\\
		&= \lim_{ n \to \infty } \| e V_{n} \|^{2} \\
		&= \lim_{ n \to \infty } \| e e_{n} e \| \\
		&= 0. 
\end{align*}
Since
\begin{align*}
\left\| 1_{ \multialg{ \mathfrak{A} \otimes \K } } - V_{n}^{*} y\left(\sum_{ i = 1}^{m} p a_{i} x b_{i} p\right) V_{n} \right\| = \left\|  1_{ \multialg{ \mathfrak{A} \otimes \K } } - V_{n}^{*} p V_{n} \right\|, 
\end{align*}
we have that 
\begin{align*}
\lim_{ n \to \infty } \left\| 1_{ \multialg{ \mathfrak{A} \otimes \K } } - V_{n}^{*} y\left(\sum_{ i = 1}^{m} pa_{i} x b_{i} p\right) V_{n} \right\|  = 0.
\end{align*}
Therefore, $x$ is norm-full in $\multialg{ \mathfrak{A} \otimes \K }$.
\end{proof}

\begin{propo}\label{p:sepid}
Let $\mathfrak{A}$ be a $C^{*}$-algebra and let $\mathfrak{I}$ be an ideal of $\mathfrak{A}$ such that $\mathfrak{I}$ is stable and $\mathfrak{I}$ has an approximate identity consisting of projections.  Suppose
\begin{equation*}
\mathfrak{e} : \ 0 \to \mathfrak{I} \to \mathfrak{A} \to \mathfrak{A} / \mathfrak{I} \to 0
\end{equation*}
is a full extension.  Then $\mathfrak{e}$ is stenotic.  Consequently, if $\mathfrak{A} / \mathfrak{I}$ is simple, then $\mathfrak{I}$ is the largest proper ideal of $\mathfrak{A}$.
\end{propo}

\begin{proof}
Let $\mathfrak{D}$ be an ideal of $\mathfrak{A}$.   Suppose $\mathfrak{D}$ is not a subset of $\mathfrak{I}$.  Hence, there exists $a \in \mathfrak{D} \setminus \mathfrak{I}$.  Note that we may assume that $a \in \mathfrak{A}_{+}$.  Since $\mathfrak{e}$ is a full extension and $a \notin \mathfrak{I}$, $\tau_{ \mathfrak{e} } ( \pi_{ \mathfrak{I} , \mathfrak{A} } ( a ) )$ is norm-full in $\corona{ \mathfrak{I} }$.  By Lemma~\ref{l:fullcorona}, $\sigma_{ \mathfrak{e} } ( a )$ is norm-full in $\multialg{ \mathfrak{I} }$.

Let $b \in \mathfrak{I}$ and let $\epsilon > 0$.  Then there exist $x_{i}, y_{i} \in \multialg{ \mathfrak{I} }$ such that 
\begin{equation*}
\left\| \theta_{ \mathfrak{I} } ( b ) - \sum_{ i = 1}^{n} x_{i}  \sigma_{\mathfrak{ e } } ( a ) y_{i}  \right\| < \frac{\epsilon}{2}.
\end{equation*}
Choose $e \in \mathfrak{I}$ such that $\| \theta_{ \mathfrak{I} } ( b ) - \theta_{ \mathfrak{I} } ( e b e ) \| < \frac{\epsilon}{2}$ and $\| e \| \leq 1$.  Then 
\begin{equation*}
\left\| \theta_{ \mathfrak{I} } ( b ) - \sum_{ i = 1}^{n} \sigma_{\mathfrak{ e } } ( e ) x_{i} \sigma_{ \mathfrak{e} } ( a ) y_{i} \sigma_{ \mathfrak{e} } (e ) \right\| < \epsilon.
\end{equation*}
Note that there exist $\alpha_{i}, \beta_{i } \in \mathfrak{I}$ such that $\sigma_{ \mathfrak{e} } ( e ) x_{i} = \theta_{ \mathfrak{I} } ( \alpha_{i} ) = \sigma_{ \mathfrak{e} } ( \alpha_{i} )$ and $y_{i}  \sigma_{ \mathfrak{e} } ( e ) = \theta_{ \mathfrak{I} } ( \beta_{i} ) = \sigma_{ \mathfrak{e} } ( \beta_{i} )$.  Hence, 
\begin{equation*}
\left\| \sigma_{ \mathfrak{e} } ( b ) - \sum_{ i = 1}^{n} \sigma_{\mathfrak{ e } } ( \alpha_{i} a \beta_{i} ) \right\| < \epsilon.
\end{equation*}
Since $\sigma_{ \mathfrak{e} }$ is an injective $*$-homomorphism, 
\begin{equation*}
\left\|  b  - \sum_{ i = 1}^{n} \alpha_{i} a \beta_{i} \right\| < \epsilon.
\end{equation*}
Therefore, $b$ is in the ideal of $\mathfrak{A}$ generated by $a$.  Since $a \in \mathfrak{D}$ and $\mathfrak{D}$ is an ideal of $\mathfrak{A}$, $b \in \mathfrak{D}$.  Thus, $\mathfrak{I} \subseteq \mathfrak{D}$.
\end{proof}

\begin{defin}
Let $\mathfrak{A}$ be a $C^{*}$-algebra.  We say that $\mathfrak{A}$ has \emph{cancellation of projections} if for each $n,m, k \in \N$ and for all projections $p \in \mathsf{M}_{n} ( \mathfrak{A} )$, $q \in \mathsf{M}_{m} ( \mathfrak{A} )$, and $e \in \mathsf{M}_{k} ( \mathfrak{A} )$,
\begin{align*}
p \oplus e \sim q \oplus e \quad \implies \quad p \sim q.
\end{align*}
\end{defin}

By Corollary~6.5.7 of \cite{blackadarB}, if $\mathfrak{A}$ has real rank zero, then $\mathfrak{A}$ has cancellation of projections if and only if $\mathfrak{A}$ has stable rank one.  Thus, if $\mathfrak{A}$ has real rank zero and $\mathfrak{A}$ has cancellation of projections, then $K_{0} ( \mathfrak{A} )_{+} \neq K_{0} ( \mathfrak{A} )$.  

\begin{lemma}\label{l:str1}
Let $\mathfrak{A}$ be a $C^{*}$-algebra with cancellation of projections.  If $\mathfrak{I}$ is an ideal of $\mathfrak{A}$ and $p$ and $q$ are projections in $\mathfrak{I}$, then $[ \iota_{\mathfrak{I} , \mathfrak{A} } (p) ] \leq [ \iota_{ \mathfrak{I} , \mathfrak{A} } (q) ]$ in $K_{0} ( \mathfrak{A} )$ if and only if $p \lesssim q$ in $\mathfrak{I}$.  Consequently, if the real rank of $\mathfrak{A}$ is zero, then $K_{0} ( \iota_{ \mathfrak{I}, \mathfrak{A} } )$ is injective and $x \in K_{0} ( \mathfrak{I} )_{+}$ whenever $K_{0} ( \iota_{\mathfrak{I}, \mathfrak{A} } ) (x) \in K_{0} ( \mathfrak{A} )_{+}$.  
\end{lemma}

The following definition is a slight modification of the definition of \emph{lexicographic ordering} given in \cite{dh:extAF}.  In \cite{dh:extAF}, lexicographic ordering referred to the ordering of an exact sequence of dimension groups.  We extend the definition for general pre-ordered groups.  

\begin{defin}
Let $(G, G_{+} )$, $( H, H_{+})$, and $( K, K_{+} )$ be pre-ordered groups.  Let $\ftn{ \alpha }{ H }{ G }$ and $\ftn{ \beta }{ G }{ K }$ be group homomorphisms.  The sequence of pre-ordered groups
\begin{align}\label{exordered}
( H , H_{+} ) \overset{ \alpha }{ \to } ( G , G_{+} ) \overset{ \beta }{ \to } (K, K_{+} )
\end{align}
is \emph{exact} if $\alpha( H_{+} ) \subseteq G_{+}$, $\beta( G_{+} ) \subseteq K_{+}$, and $H \overset{ \alpha }{ \to } G \overset{ \beta }{ \to } K$ is an exact sequence of groups.  We say that the exact sequence in (\ref{exordered}) of pre-ordered groups is \emph{lexicographic} if
\begin{align*}
G_{+} = \beta^{-1} ( K_{+} \setminus \{ 0 \} ) \sqcup \alpha (  H_{+} ).
\end{align*}
Note that the union is always disjoint.
\end{defin}

\begin{defin}\label{d:klexi}
For a $C^{*}$-algebra $\mathfrak{A}$, set 
\begin{align*}
K_{0} ( \mathfrak{A} )_{+} =  \setof{ [ p ] }{ \text{$p$ is a projection in $\mathsf{M}_{n} ( \mathfrak{A} )$ for some $n \in \N$} }
\end{align*}
and
\begin{align*}
K_{0} ( \mathfrak{A} )_{++} = \setof{ [ p ] }{ \text{$p$ is a norm-full projection in $\mathfrak{A}$} }.
\end{align*}

Let $\mathfrak{A}$ be a $C^{*}$-algebra and let $\mathfrak{I}$ be an ideal of $\mathfrak{A}$.  Let 
\begin{align*}
\mathfrak{e} : 0 \to \mathfrak{I} \to  \mathfrak{A} \to \mathfrak{A} / \mathfrak{I} \to 0
\end{align*}   
be an extension of $C^{*}$-algebras.  We say that $\mathfrak{e}$ is \emph{$K$-lexicographic} if
\begin{itemize}
\item[(1)] the exact sequence of pre-ordered groups
\begin{align*}
( K_{0} ( \mathfrak{I} ) , K_{0} ( \mathfrak{I} )_{+} ) \to  ( K_{0} ( \mathfrak{A} ) , K_{0} ( \mathfrak{A} )_{+} ) \to ( K_{0} ( \mathfrak{A} / \mathfrak{I} ) , K_{0} ( \mathfrak{A} / \mathfrak{I} )_{+} )
\end{align*}
induced by $\mathfrak{e}$ is lexicographic whenever $K_{0} ( \mathfrak{A} / \mathfrak{I} ) \neq K_{0} ( \mathfrak{A} / \mathfrak{I} )_{+}$ and

\item[(2)] $K_{0} ( \mathfrak{A} ) = K_{0} ( \mathfrak{A} )_{+} = K_{0} ( \mathfrak{A} )_{++}$ whenever $K_{0} ( \mathfrak{A} / \mathfrak{I} ) = K_{0} ( \mathfrak{A} / \mathfrak{I} )_{+} = K_{0} ( \mathfrak{A} / \mathfrak{I} )_{++}$.
\end{itemize}
\end{defin}

\begin{defin}
A $C^{*}$-algebra $\mathfrak{A}$ is said to satisfy the \emph{corona factorization property} if $P \sim 1_{ \multialg{ \mathfrak{A} \otimes \K } }$ for all norm-full projection $P \in \multialg{ \mathfrak{A} \otimes \K }$.
\end{defin}

The following results are due to Kucerovsky
and Ng (see \cite{NgCFP} and \cite{KucNgCFPdef}).  They show that many $C^{*}$-algebras satisfy the corona factorization property.  

\begin{theor}
Let $\mathfrak{A}$ be a unital separable simple \cstar-algebra.

\begin{enumerate}

\item If $\mathfrak{A}$ is exact, $\mathfrak{A}$ has real rank zero and stable rank one, and $K_{0} ( \mathfrak{A} )$ is weakly unperforated, then $\mathfrak{A} \otimes \K$ has the corona factorization property.

\item If $\mathfrak{A}$ is purely infinite, then $\mathfrak{A} \otimes \K$ has the corona factorization property.
\end{enumerate}
\end{theor}

\begin{defin}
A projection $p \in \mathfrak{A}$ is called \emph{properly infinite} if $p \oplus p \lesssim p$.
\end{defin}

We will make use of the following proposition which is contained in the proof of Proposition~1.5 of \cite{kthypureinf}.  See Theorem~7 and Lemma~8 of \cite{segr} for details.

\begin{propo}\label{p:cuntz}
Let $\mathfrak{A}$ be a $C^{*}$-algebra with a norm-full properly infinite projection.  Then $K_{0} ( \mathfrak{A} ) = K_{0} ( \mathfrak{A} )_{+} = K_{0} ( \mathfrak{A} )_{++}$.  Moreover, if $p$ and $q$ are norm-full properly infinite projections, then there exists a projection $e \in \mathfrak{A}$ such that $p \sim e < q$.
\end{propo}

The next proposition gives a necessary $K$-theoretical condition for an extension to be full.

\begin{propo}\label{p:fullkthy}
Let $\mathfrak{A}$ be a separable $C^{*}$-algebra with real rank zero.  Let $\mathfrak{I}$ be an ideal of $\mathfrak{A}$ such that $\mathfrak{I}$ satisfies the corona factorization property, and $\mathfrak{I}$ is stable.  Suppose the extension
\begin{equation*}
\mathfrak{e} : 0 \to \mathfrak{I} \to \mathfrak{A} \to \mathfrak{A} / \mathfrak{I} \to 0
\end{equation*}
is full. Suppose for each ideal $\mathfrak{D}$ of $\mathfrak{A}$ with $\mathfrak{I} \subsetneq \mathfrak{D}$, $\mathfrak{D} / \mathfrak{I}$ has cancellation of projections or has a norm-full properly infinite projection.  Then for each ideal $\mathfrak{D}$ of $\mathfrak{A}$ with $\mathfrak{I} \subsetneq \mathfrak{D}$, 
\begin{equation*}
0 \to \mathfrak{I} \to \mathfrak{D} \to \mathfrak{D} / \mathfrak{I} \to 0
\end{equation*}
is $K$-lexicographic.  
\end{propo}

\begin{proof}
First note that since $\mathfrak{e}$ is a full extension and $\mathfrak{I}$ is stable, by Proposition~1.6 of \cite{ERRshift}, 
\begin{align*}
\mathfrak{e}^{s} : 0 \to \mathfrak{I} \otimes \K \to \mathfrak{A} \otimes \K \to ( \mathfrak{A} / \mathfrak{I} ) \otimes \K \to 0
\end{align*}
is a full extension.  Let $\mathfrak{D}$ be an ideal of $\mathfrak{A}$ with $\mathfrak{I} \subsetneq \mathfrak{D}$.  Suppose $\mathfrak{D} / \mathfrak{I}$ has cancellation of projections.  Since $K_{0} ( \mathfrak{D} / \mathfrak{I} )_{+} \neq K_{0} ( \mathfrak{D} / \mathfrak{I} )$, to prove that 
\begin{align*}
0 \to \mathfrak{I} \to \mathfrak{D} \to \mathfrak{D} / \mathfrak{I} \to 0
\end{align*}  
is $K$-lexicographic, we need to show that
\begin{align*}
( K_{0} ( \mathfrak{I} ) , K_{0} ( \mathfrak{I} )_{+} ) \to ( K_{0} ( \mathfrak{D} ) , K_{0} ( \mathfrak{D} )_{+} ) \to ( K_{0} ( \mathfrak{D} / \mathfrak{I} ) , K_{0} ( \mathfrak{D} / \mathfrak{I} )_{+} )
\end{align*}
is lexicographic, i.e., 
\begin{align*}
K_{0} ( \mathfrak{D} )_{+} = K_{0}(\iota_{ \mathfrak{I}, \mathfrak{D} } ) ( K_{0} ( \mathfrak{I} )_{+} ) \sqcup \setof{ x \in K_{0} ( \mathfrak{D} ) }{ K_{0} ( \pi_{ \mathfrak{I} , \mathfrak{D} } )( x ) > 0 }.
\end{align*}

Set $Y =  K_{0}(\iota_{ \mathfrak{I}, \mathfrak{D} } ) ( K_{0} ( \mathfrak{I} )_{+} ) \sqcup \setof{ x \in K_{0} ( \mathfrak{D} ) }{ K_{0} ( \pi_{ \mathfrak{I} , \mathfrak{D} } )( x ) > 0 }$.  By Lemma~\ref{l:kthyrr0}, $K_{0} ( \mathfrak{D} )_{+} \subseteq Y$.  It is clear that $K_{0}(\iota_{ \mathfrak{I}, \mathfrak{D} } ) ( K_{0} ( \mathfrak{I} )_{+} ) \subseteq K_{0} ( \mathfrak{D} )_{+}$.  Let $x \in K_{0} ( \mathfrak{D} )$ such that $K_{0} ( \pi_{\mathfrak{I}, \mathfrak{D} } ) ( x ) > 0$.  By Lemma~\ref{l:kthyrr0}, there exist projections $p$ and $q$ in $\mathfrak{D} \otimes \K$ such that $x = [ p ] - [ q ]$.  Since $K_{0} ( \pi_{\mathfrak{I},  \mathfrak{D} } )( x ) > 0$ and since $\mathfrak{D} / \mathfrak{I}$ has cancellation of projections, there exist nonzero orthogonal projections $e , q_{1} \in ( \mathfrak{D}  / \mathfrak{I} ) \otimes \K$ such that $q_{1} \sim \pi_{ \mathfrak{I}, \mathfrak{D} }^{s} (q)$ and $\pi_{ \mathfrak{I} ,\mathfrak{D} }^{s}( p ) \sim (q_{1} + e )$.  Hence, $\pi_{ \mathfrak{I}, \mathfrak{D} }^{s} ( q ) \sim e' < \pi_{ \mathfrak{I}, \mathfrak{D} }^{s} ( p )$ in $( \mathfrak{D} / \mathfrak{I} )\otimes \K$.  By Lemma~\ref{l:strict}, there exists a partial isometry $v_{1}$ in $\mathfrak{D} \otimes \K$ such that $v_{1}^{*} v_{1} \leq q$, $q - v_{1}^{*} v_{1} \in \mathfrak{I} \otimes \K$, $v_{1} v_{1}^{*} \leq p$, and $p - v_{1} v_{1}^{*} \notin \mathfrak{I} \otimes \K$.  Since $\mathfrak{e}^{s}$ is a full extension, $\tau_{ \mathfrak{e}^{s} } ( \pi_{ \mathfrak{I}, \mathfrak{A} }^{s} (p - v_{1} v_{1}^{*} ) )$ is a norm-full projection in $\corona{ \mathfrak{I} \otimes \K }$.  By Lemma~\ref{l:fullcorona}, $\sigma_{ \mathfrak{e}^{s} } ( p - v_{1} v_{1}^{*} )$ is a norm-full projection in $\multialg{ \mathfrak{I} \otimes \K }$.  Hence, $\sigma_{ \mathfrak{e}^{s} } ( p - v_{1} v_{1}^{*} ) \sim 1_{ \multialg{ \mathfrak{I} \otimes \K } }$ in $\multialg{ \mathfrak{I} \otimes \K }$ since $\mathfrak{I} \otimes \K$ satisfies the corona factorization property.  Therefore, $\theta_{ \mathfrak{I} \otimes \K } ( q - v_{1}^{*}v_{1} ) \lesssim \sigma_{ \mathfrak{e}^{s} } ( p - v_{1}v_{1}^{*} )$ in $\multialg{ \mathfrak{I} \otimes \K }$.  Since $\theta_{ \mathfrak{I} \otimes \K } ( q - v_{1}^{*}v_{1} ) \in \theta_{ \mathfrak{I} \otimes \K } ( \mathfrak{I} \otimes \K )$ and $\theta_{ \mathfrak{I} \otimes \K } ( \mathfrak{I} \otimes \K )$ is an ideal of $\multialg{ \mathfrak{I} \otimes \K }$, by Lemma~\ref{l:her}, there exists $v_{2} \in \mathfrak{I} \otimes \K$ such that $\theta_{ \mathfrak{I} \otimes \K } ( v_{2}^{*}v_{2} ) =  \theta_{ \mathfrak{I} \otimes \K } ( q- v_{1}^{*} v_{1} )$ and $\theta_{ \mathfrak{I} \otimes \K } ( v_{2} v_{2}^{*} ) \leq \sigma_{ \mathfrak{e}^{s} } ( p - v_{1}v_{1}^{*} )$.  By Lemma~\ref{l:essentialmv}, $v_{2} v_{2}^{*} \leq p - v_{1} v_{1}^{*}$.  Set $v = v_{1} + v_{2}$.  Then $v \in \mathfrak{D} \otimes \K$, $v^{*} v = v_{1}^{*} v_{1} + v_{2}^{*} v_{2} = q$, and $vv^{*} = v_{1}v_{1}^{*} + v_{2}v_{2}^{*} \leq p$.  Hence, $x \in K_{0} ( \mathfrak{D} )_{+}$.  Thus, $K_{0} ( \mathfrak{D} )_{+} = Y$, which implies that 
\begin{equation*}
0 \to \mathfrak{I} \to \mathfrak{D} \to \mathfrak{D} / \mathfrak{I} \to 0 
\end{equation*}
is $K$-lexicographic.

Suppose $\mathfrak{D} / \mathfrak{I}$ has a norm-full properly infinite projection.  Then by Proposition~\ref{p:cuntz}, $K_{0} ( \mathfrak{D} / \mathfrak{I} ) = K_{0} ( \mathfrak{D} / \mathfrak{I} )_{+} = K_{0} ( \mathfrak{D} / \mathfrak{I} )_{++}$.  So, to prove that 
\begin{align*}
0 \to \mathfrak{I} \to \mathfrak{D} \to \mathfrak{D} / \mathfrak{I} \to 0
\end{align*}  
is $K$-lexicographic, we need to show that $K_{0} ( \mathfrak{D} )_{++} = K_{0} ( \mathfrak{D} )_{+} = K_{0} ( \mathfrak{D} )$.  To show this, we first show that for all norm-full projections $p, q$ in $\mathfrak{D} \otimes \K$ with $\pi^{s}_{ \mathfrak{I} , \mathfrak{D} } (p)$ and $\pi^{s} _{ \mathfrak{I} , \mathfrak{D} } (q)$ properly infinite, $q \lesssim p$ in $\mathfrak{D} \otimes \K$.  

Let $p$ and $q$ be norm-full projections in $\mathfrak{D} \otimes \K$ such that $\pi^{s}_{ \mathfrak{I} , \mathfrak{D} } (p)$ and $\pi^{s} _{ \mathfrak{I} , \mathfrak{D} } (q)$ are properly infinite projections in $\left( \mathfrak{D} / \mathfrak{I} \right) \otimes \K$.  By Proposition~\ref{p:cuntz}, there exists a projection $e$ in $( \mathfrak{D} / \mathfrak{I} ) \otimes \K$ such that $\pi_{ \mathfrak{I}, \mathfrak{D} }^{s} ( q ) \sim e < \pi_{ \mathfrak{I}, \mathfrak{D} }^{s} ( p )$ in $( \mathfrak{D} / \mathfrak{I} ) \otimes \K$.  By Lemma~\ref{l:strict}, there exists $v_{1} \in \mathfrak{D} \otimes \K$ such that $v_{1}^{*} v_{1} \leq q$, $v_{1} v_{1}^{*} \leq p$, $\pi_{ \mathfrak{I},  \mathfrak{D} }^{s} ( v_{1}^{*} v_{1} ) = \pi_{ \mathfrak{I}, \mathfrak{D} }^{s} ( q )$ and $\pi_{ \mathfrak{I}, \mathfrak{D} }^{s} ( v_{1} v_{1}^{*} ) \neq \pi_{ \mathfrak{D} }^{s} ( p )$.  Since $\mathfrak{e}^{s}$ is a full extension, $\tau_{ \mathfrak{e}^{s} } ( \pi_{ \mathfrak{I}, \mathfrak{A} }^{s} ( p - v_{1} v_{1}^{*} ) )$ is a norm-full projection in $\corona{ \mathfrak{I} \otimes \K }$.  By Lemma~\ref{l:fullcorona}, $\sigma_{ \mathfrak{e}^{s} } ( p - v_{1} v_{1}^{*} )$ is a norm-full projection in $\multialg{ \mathfrak{I} \otimes \K }$.  Since $\mathfrak{I} \otimes \K$ has the corona factorization property, there exists $w \in \multialg{ \mathfrak{I} \otimes \K }$ such that $ww^{*} = \sigma_{ \mathfrak{e}^{s} } ( p - v_{1} v_{1}^{*} )$ and $w^{*} w = 1_{ \multialg{ \mathfrak{I} \otimes \K } }$.  Let $v_{2} \in \mathfrak{I} \otimes \K$ such that $\theta_{ \mathfrak{I} \otimes \K } ( v_{2} ) =  w\theta_{ \mathfrak{I} \otimes \K } ( q - v_{1}^{*} v_{1})$.  Since $\theta_{ \mathfrak{I} \otimes \K } ( v_{2}^{*} v_{2} ) = \theta_{ \mathfrak{I} \otimes \K } ( q - v_{1}^{*} v_{1} )$ and $\sigma_{ \mathfrak{e} } ( v_{2} v_{2}^{*} ) = \theta_{ \mathfrak{I} \otimes \K } ( v_{2} v_{2}^{*} ) \leq \sigma_{ \mathfrak{e} } ( p - v_{1} v_{1}^{*} )$, we have that $v_{2}^{*} v_{2} = q - v_{1} v_{1}^{*}$ and by Lemma~\ref{l:essentialmv}, $v_{2} v_{2}^{*} \leq p - v_{1} v_{1}^{*}$.  Thus, $v = v_{1} + v_{2} \in \mathfrak{D} \otimes \K$ such that $v^{*} v = q$ and $v v^{*} \leq p$.  We have proved our claim.  

Let $q$ be a norm-full properly infinite projection in $\mathfrak{D} / \mathfrak{I}$.  Since $\mathfrak{D}$ has real rank zero, by  \cite[Theorem~3.14]{realrank}, there exists a projection $p$ in $\mathfrak{D}$ such that $\pi_{ \mathfrak{I} , \mathfrak{D} } (p) = q$.  Since $\mathfrak{e}$ is full, by Proposition~\ref{p:sepid}, $\mathfrak{e}$ is stenotic.  Hence, $p$ is a norm-full projection in $\mathfrak{D}$.  Set $p_{1} = p \otimes e_{11}$ and $p_{2} = p \otimes e_{22}$ in $\mathfrak{D} \otimes \K$.  Then $p_{1}$ and $p_{2}$ are orthogonal projections with $p_{1} \sim p_{2}$.  Note that $p_{1} + p_{2}$ and $p_{1}$ are norm-full projections in $\mathfrak{D} \otimes \K$ and $\pi_{\mathfrak{I} , \mathfrak{D} }^{s} ( p _{1} + p_{2} )$ and $\pi_{ \mathfrak{I} , \mathfrak{D} }^{s} ( p_{1} )$ are norm-full properly infinite projections in $( \mathfrak{D} / \mathfrak{I} ) \otimes \K$.  Hence, by the above claim, $p _{ 1} + p_{2} \lesssim p_{1}$.  Therefore, $p \oplus p \lesssim p$.  We have just shown that $\mathfrak{D}$ has a norm-full properly infinite projection.  By Proposition~\ref{p:cuntz}, $K_{0} ( \mathfrak{D}  ) = K_{0} ( \mathfrak{D}  )_{+} = K_{0} ( \mathfrak{D} )_{++}$.  Hence,
\begin{equation*}
0 \to \mathfrak{I} \to \mathfrak{D} \to \mathfrak{D} / \mathfrak{I} \to 0 
\end{equation*}
is $K$-lexicographic.
\end{proof}

We will show that the converse of the above result is true when $\mathfrak{A}$ has the \emph{stable weak cancellation property}.

\begin{defin}
A $C^{*}$-algebra $\mathfrak{A}$ has \emph{weak cancellation} if any pair of projections $p$ and $q$ in $\mathfrak{A}$ that generate the same closed ideal $\mathfrak{I}$ in $\mathfrak{A}$ and have the same image in $K_{0} ( \mathfrak{I} )$ must be Murray-von Neumann equivalent.  If $\mathsf{M}_{n} ( \mathfrak{A} )$ has weak cancellation for every $n$, then we say that $\mathfrak{A}$ has \emph{stable weak cancellation}.
\end{defin}

Note that $\mathfrak{A}$ has stable weak cancellation if and only if $\mathfrak{A} \otimes \K$ has weak cancellation.  Ara, Moreno, and Pardo in \cite{amp:nonstablekthy} showed that every graph $C^{*}$-algebra has stable weak cancellation.  It is an open question if every real rank zero $C^{*}$-algebra has stable weak cancellation.

\begin{lemma}\label{l:swkc}
Let $\mathfrak{A}$ be a $C^{*}$-algebra with real rank zero.  Then $\mathfrak{I}$ and $\mathfrak{A} / \mathfrak{I}$ have stable weak cancellation if and only if $\mathfrak{A}$ has stable weak cancellation.  
\end{lemma}

\begin{proof}
Suppose $\mathfrak{I}$ and $\mathfrak{A} / \mathfrak{I}$ have stable weak cancellation.  Then, by Proposition~2.1 of \cite{amp:nonstablekthy}, $\mathfrak{I}$ and $\mathfrak{A} / \mathfrak{I}$ are separative.  Since $\mathfrak{A}$ has real rank zero, by Theorem~7.5 of \cite{agop:exhangerings}, $\mathfrak{A}$ is separative.  Thus, by Proposition~2.1 of \cite{amp:nonstablekthy}, $\mathfrak{A}$ has stable weak cancellation.

Suppose $\mathfrak{A}$ has stable weak cancellation.  By Theorem~7.5 of \cite{agop:exhangerings}, $\mathfrak{I}$ and $\mathfrak{A} / \mathfrak{I}$ are separative.  Hence, by Proposition~2.1 of \cite{amp:nonstablekthy}, $\mathfrak{I}$ and $\mathfrak{A} / \mathfrak{I}$ have stable weak cancellation.
\end{proof}

\begin{lemma}\label{l:sepid}
Let $\mathfrak{A}$ be a $C^{*}$-algebra and let $\mathfrak{I}$ be an ideal of $\mathfrak{A}$.  Suppose
\begin{equation*}
\mathfrak{e} : 0 \to \mathfrak{I} \to \mathfrak{A} \to \mathfrak{A} / \mathfrak{I} \to 0
\end{equation*}
is stenotic.  Then for all $p, q \in \mathfrak{A} \setminus \mathfrak{I}$, $p$ and $q$ generate the same ideal of $\mathfrak{A}$ if and only if $\pi_{ \mathfrak{I}, \mathfrak{A} } (p)$ and $\pi_{ \mathfrak{I} , \mathfrak{A} } (q)$ generate the same ideal of $\mathfrak{A} / \mathfrak{I}$.
\end{lemma}

\begin{proof}
The ``only if'' direction follows from the fact that $\pi_{ \mathfrak{I}, \mathfrak{A} }$ is surjective.  We prove the ``if'' direction.  Let $\mathfrak{D}_{1}$ be the ideal generated by $p$ and let $\mathfrak{D}_{2}$ be the ideal generated by $q$.  Note that $\mathfrak{I} \subseteq \mathfrak{D}_{i}$ since $p, q \in \mathfrak{A} \setminus \mathfrak{I}$ and $\mathfrak{e}$ is stenotic.  Let $x \in \mathfrak{D}_{i}$.  Then $\pi_{ \mathfrak{I}, \mathfrak{A} } (x) \in \pi_{ \mathfrak{I}, \mathfrak{A} } ( \mathfrak{D}_{j} )$.  Hence, there exists $z \in \mathfrak{D}_{j}$ such that $x - z \in \mathfrak{I}$.  Since $\mathfrak{I} \subseteq \mathfrak{D}_{j}$, we have that $x \in \mathfrak{D}_{j}$.  So, $\mathfrak{D}_{i} \subseteq \mathfrak{D}_{j}$. 
\end{proof}

\begin{lemma}\label{l:rr0full}
Let $\mathfrak{A}$ be a $C^{*}$-algebra such that $\mathfrak{A}$ has real rank zero and let $\mathfrak{I}$ be an ideal of $\mathfrak{A}$.  If $\tau_{\mathfrak{e}} ( \pi_{ \mathfrak{I}, \mathfrak{A} } ( p) )$ is a norm-full projection in $\corona{\mathfrak{I}}$ for all projections $p \in \mathfrak{A} \setminus \mathfrak{I}$, then 
\begin{align*}
\mathfrak{e} : 0 \to \mathfrak{I} \to \mathfrak{A} \to \mathfrak{A} / \mathfrak{I} \to 0
\end{align*}
is a full extension.
\end{lemma}

\begin{proof}
Since $\mathfrak{A}$ has real rank zero, by the results in \cite{realrank}, $\mathfrak{A} / \mathfrak{I}$ has real rank zero and every projection in $\mathfrak{A} / \mathfrak{I}$ lifts to a projection in $\mathfrak{A}$.  Let $a$ be a nonzero element in $\mathfrak{A} / \mathfrak{I}$.  Let $\mathfrak{D}_{1}$ be the ideal of $\mathfrak{A} / \mathfrak{I}$ generated by $a$ and let $\mathfrak{D}_{2}$ be the ideal of $\corona{\mathfrak{I}}$ generated by $\tau_{ \mathfrak{e} } ( a )$.  Note that $\tau_{ \mathfrak{e} } ( \mathfrak{D}_{1} ) \subseteq \mathfrak{D}_{2}$.  Since $\mathfrak{A}$ has real rank zero, there exists a projection $p \in \mathfrak{A} \setminus \mathfrak{I}$ such that $\pi_{ \mathfrak{I}, \mathfrak{A} } ( p ) \in \mathfrak{D}_{1}$.  Hence, $\tau_{ \mathfrak{e} } ( \pi_{ \mathfrak{I}, \mathfrak{A} } (p) ) \in \mathfrak{D}_{2}$.  Since $\tau_{ \mathfrak{e} } ( \pi_{ \mathfrak{I}, \mathfrak{A} } (p) )$ is a norm-full projection in $\corona{\mathfrak{I} }$, we have that $\mathfrak{D}_{2} = \corona{ \mathfrak{I} }$.  Hence, $\mathfrak{e}$ is a full extension.
\end{proof}

\begin{lemma}\label{l:fullher}
Let $\mathfrak{A}$ be a $C^{*}$-algebra and let $a \in \mathfrak{A}_{+}$.  If $a$ is norm-full in $\mathfrak{A}$ and $\mathfrak{I}$ is an ideal of $\mathfrak{A}$, then $\overline{ a \mathfrak{I} a }$ is norm-full in $\mathfrak{I}$.
\end{lemma}

\begin{proof}
Let $b \in \mathfrak{I}$ be a non-zero positive element and $\| b \| \leq 1$.  Let $\epsilon > 0$.  Since $a$ is norm-full in $\mathfrak{A}$, there exist  $n \in \N$ and $x_{i}, y_{i} \in \mathfrak{A}$ such that 
\begin{align*}
\left\| b^{ \frac{1}{2} } - \sum_{ i = 1}^{n} x_{i} a y_{i} \right\| < \frac{\epsilon}{2}.
\end{align*}
Then 
\begin{align*}
\left\| b - \sum_{ i = 1}^{n} b^{\frac{1}{4}}x_{i} a y_{i} b^{ \frac{1}{4} } \right\| < \frac{\epsilon}{2}. 
\end{align*}
Choose $e \in \mathfrak{I}$ such that 
\begin{align*}
 \left\| \sum_{ i = 1}^{n} b^{\frac{1}{4}}x_{i} a y_{i} b^{ \frac{1}{4} } - \sum_{ i = 1}^{n} b^{ \frac{1}{4} } x_{i} a^{ \frac{1}{2} } e a^{ \frac{1}{2} } y_{i} b^{ \frac{1}{4} }  \right\| < \frac{\epsilon}{2}.
\end{align*}
This can be done since $b^{ \frac{1}{4} } x_{i} a^{ \frac{1}{2} } \in \mathfrak{I}$.  Thus,
\begin{align*}
&\left\| b -  \sum_{ i = 1}^{n} b^{ \frac{1}{4} } x_{i} a^{ \frac{1}{2} } ea^{ \frac{1}{2} } y_{i} b^{ \frac{1}{4} } \right\| \\
&\qquad \leq \left\| b - \sum_{ i = 1}^{n} b^{\frac{1}{4}}x_{i} a y_{i} b^{ \frac{1}{4} } \right\| + \left\| \sum_{ i = 1}^{n} b^{\frac{1}{4}}x_{i} a y_{i} b^{ \frac{1}{4} }  -  \sum_{ i = 1}^{n} b^{ \frac{1}{4} } x_{i} a^{ \frac{1}{2} } e a^{ \frac{1}{2} } y_{i} b^{ \frac{1}{4} } \right\| \\
&\qquad  < \epsilon.
\end{align*}
Note that $\sum_{ i = 1}^{n} b^{ \frac{1}{4} } x_{i} a^{ \frac{1}{2} } e a^{ \frac{1}{2} } y_{i} b^{ \frac{1}{4} }$ is in the ideal of $\mathfrak{I}$ generated by $\overline{ a \mathfrak{I} a }$.  We have just shown that every non-zero positive element with norm less than or equal to 1 is in the ideal of $\mathfrak{I}$ generated by $\overline{ a \mathfrak{I} a }$.

Let $b \in \mathfrak{I}$ be non-zero.  Since $b$ and $\frac{ b^{*}b }{ \| b \|^{2} }$ generate the same ideal, we have that $b$ is in the ideal of $\mathfrak{I}$ generated by $\overline{ a \mathfrak{I} a }$.
\end{proof}

We are now ready to prove the converse of Proposition~\ref{p:fullkthy} with the additional assumption that the algebra has the stable weak cancellation property. 

\begin{theor}\label{t:full}
Let $\mathfrak{A}$ be a separable $C^{*}$-algebra with real rank zero.  Let $\mathfrak{I}$ be an ideal of $\mathfrak{A}$ such that $\mathfrak{I}$ satisfies the corona factorization property and $\mathfrak{I}$ is stable.  Consider the extension
\begin{equation*}
\mathfrak{e} : 0 \to \mathfrak{I} \to \mathfrak{A}  \to \mathfrak{A} / \mathfrak{I}  \to 0.
\end{equation*}
Suppose $\mathfrak{A}$ has stable weak cancellation and for all ideals $\mathfrak{D}$ of $\mathfrak{A}$ with $\mathfrak{I} \subsetneq \mathfrak{D}$, $\mathfrak{D} / \mathfrak{I}$ has cancellation of projections or $\mathfrak{D} / \mathfrak{I}$ has a norm-full properly infinite projection.  Then $\mathfrak{e}$ is full if and only if $\mathfrak{e}$ is stenotic and for all ideals $\mathfrak{D}$ of $\mathfrak{A}$ with $\mathfrak{I} \subsetneq \mathfrak{D}$, 
\begin{equation*}
0 \to \mathfrak{I} \to \mathfrak{D} \to \mathfrak{D} / \mathfrak{I} \to 0 
\end{equation*}
is $K$-lexicographic.
\end{theor}

\begin{proof}
Note that the ``only if'' follows from Proposition~\ref{p:sepid} and Proposition~\ref{p:fullkthy}.  Suppose $\mathfrak{e}$ is stenotic and for all ideals $\mathfrak{D}$ of $\mathfrak{A}$ with $\mathfrak{I} \subsetneq \mathfrak{D}$, 
\begin{equation*}
0 \to \mathfrak{I} \to \mathfrak{D} \to \mathfrak{D} / \mathfrak{I} \to 0 
\end{equation*}
is $K$-lexicographic.

Let $p$ be a projection in $\mathfrak{A} \setminus \mathfrak{I}$.  Let $e$ and $f$ be projections in $\mathfrak{I}$ such that $f \leq p$.  Let $\mathfrak{D}$ be the ideal of $\mathfrak{A}$ generated by $p$.  Since $p $ is not in $\mathfrak{I}$, $\mathfrak{I} \subsetneq \mathfrak{D}$.  Therefore, $\mathfrak{D} / \mathfrak{I}$ has cancellation of projections or has a norm-full properly infinite projection.  

Suppose $\mathfrak{D} / \mathfrak{I}$ has cancellation of projections.  Therefore, $K_{0} ( \pi_{ \mathfrak{I} , \mathfrak{D} } ) ( [ (p - f) \otimes e_{11} ] - [ e  \otimes e_{11} ] ) = K_{0} ( \pi_{\mathfrak{I}, \mathfrak{D} } ) ( [ p \otimes e_{11} ] ) > 0$ in $K_{0} ( \mathfrak{D} / \mathfrak{I} )$ since $\mathfrak{D} / \mathfrak{I}$ has cancellation of projections.  Since 
\begin{equation*}
0 \to \mathfrak{I} \to \mathfrak{D} \to \mathfrak{D} / \mathfrak{I} \to 0 
\end{equation*}
is $K$-lexicographic and $K_{0} ( \mathfrak{D} / \mathfrak{I} )_{+} \neq K_{0} ( \mathfrak{D} / \mathfrak{I} )$, there exists a projection $q \in \mathfrak{D} \otimes \K$ such that $[ (p - f) \otimes e_{11} ] - [ e \otimes e_{11} ] = [ q ]$ in $K_{0} ( \mathfrak{D} )$.  Since $K_{0} ( \pi_{ \mathfrak{I} , \mathfrak{D} } )( [ q ] ) = K_{0} ( \pi_{ \mathfrak{I} , \mathfrak{D} } ) ( [ (p - f) \otimes e_{11} ] - [ e  \otimes e_{11} ] ) = K_{0} ( \pi_{\mathfrak{I}, \mathfrak{D} } ) ( [ p \otimes e_{11} ] )  > 0$, $q \notin \mathfrak{I} \otimes \K$.  Since $\mathfrak{D} / \mathfrak{I}$ has cancellation of projections, $\pi_{ \mathfrak{I} , \mathfrak{D} }^{s} ( q ) \sim \pi_{ \mathfrak{I} , \mathfrak{D} }^{s} ( (p - f )\otimes e_{11} ) \sim \pi_{ \mathfrak{I} , \mathfrak{D} }^{s} ( p \otimes e_{11} )$.

Note that $\mathfrak{e}^{s}$ is stenotic since $\mathfrak{e}$ is stenotic.  Choose orthogonal projections $q_{1}$ and $e_{1}$ in $\mathfrak{D} \otimes \K$ such that $q_{1} \sim q$ and $e_{1} \sim e \otimes e_{11}$.  Since $\pi_{ \mathfrak{I} , \mathfrak{D} }^{s} ( q ) \sim \pi_{ \mathfrak{I} , \mathfrak{D} }^{s} ( (p - f )\otimes e_{11} ) \sim \pi_{ \mathfrak{I} , \mathfrak{D} }^{s} ( p \otimes e_{11} )$, the ideal of $(\mathfrak{A} / \mathfrak{I} ) \otimes \K$ generated by $\pi_{\mathfrak{I} , \mathfrak{A} }^{s} ( (p-f) \otimes e_{11} )$ is $( \mathfrak{D} / \mathfrak{I} ) \otimes \K$ and the ideal of $( \mathfrak{A} / \mathfrak{I} ) \otimes \K$ generated by $\pi_{ \mathfrak{I}, \mathfrak{A} }^{s} ( q_{1} + e_{1} )$ is $( \mathfrak{D} / \mathfrak{I} ) \otimes \K$.  Hence, by Lemma~\ref{l:sepid}, we have that $(p-f) \otimes e_{11}$ and $q_{1} + e_{1}$ generate the same ideal in $\mathfrak{A} \otimes \K$, which is $\mathfrak{D} \otimes \K$.  Since $[ (p-f) \otimes e_{11} ] = [ q_{1} + e_{1} ]$ in $K_{0} ( \mathfrak{D} )$ and since $\mathfrak{A}$ has stable weak cancellation, we have that $( p -f ) \otimes e_{11} \sim (q_{1} + e_{1} )$ in $\mathfrak{D} \otimes \K$.  Thus, by Lemma~\ref{l:her}, $e \otimes e_{11} \lesssim (p-f) \otimes e_{11}$ in $\mathfrak{D} \otimes e_{11}$.  Hence, $e \lesssim p-f$ in $\mathfrak{D}$. 

Since $p$ is norm-full in $\mathfrak{D}$ and $\mathfrak{I} \subseteq \mathfrak{D}$, by Lemma~\ref{l:fullher}, $p  \mathfrak{I} p$ is norm-full in $\mathfrak{I}$.  Therefore, by Lemma~\ref{lem:projfull}, $\sigma_{ \mathfrak{e} } (p)$ is a norm-full projection in $\multialg{ \mathfrak{I} }$.  We have just shown that for every projection $p \in \mathfrak{A} \setminus \mathfrak{I}$, $\sigma_{ \mathfrak{e} } ( p )$ is a norm-full projection in $\multialg{ \mathfrak{I} }$.  Therefore, by Lemma~\ref{l:rr0full}, $\mathfrak{e}$ is a full extension.

Suppose $\mathfrak{D} / \mathfrak{I}$ has a norm-full properly infinite projection.  Then by Proposition~\ref{p:cuntz}, $K_{0} ( \mathfrak{D} / \mathfrak{I} ) = K_{0} ( \mathfrak{D} / \mathfrak{I} )_{+} = K_{0} ( \mathfrak{D} / \mathfrak{I} )_{++}$.  Since $\pi_{ \mathfrak{I}, \mathfrak{A} } ( p ) = \pi_{ \mathfrak{I}, \mathfrak{A} } ( p - f )$ and $\mathfrak{e}$ is stenotic, by Lemma~\ref{l:sepid}, $p-f$ is a norm-full projection in $\mathfrak{D}$.  Note that $[ p - f ] - [ e ] \in K_{0} ( \mathfrak{D} )$.  Since 
\begin{equation*}
0 \to \mathfrak{I} \to \mathfrak{D} \to \mathfrak{D} / \mathfrak{I} \to 0 
\end{equation*}
is $K$-lexicographic and $K_{0} ( \mathfrak{D} / \mathfrak{I} ) = K_{0} ( \mathfrak{D} / \mathfrak{I} )_{+} = K_{0} ( \mathfrak{D} / \mathfrak{I} )_{++}$, there exists a norm-full projection $q \in \mathfrak{D}$ such that $[ p - f ] - [e] = [q]$ in $K_{0} ( \mathfrak{D} )$.  Note that $( p - f ) \otimes e_{11}$ and $q \otimes e_{22} + e \otimes e_{11}$ are norm-full projections in $\mathfrak{D} \otimes \K$ with $[ ( p - f ) \otimes e_{11} ] = [ q \otimes e_{22} +  e \otimes e_{11} ]$ in $K_{0} ( \mathfrak{D} )$.  Since $\mathfrak{A}$ has stable weak cancellation, $( p  -f  ) \otimes e_{11} \sim ( q \otimes e_{22} + e \otimes e_{11} )$ in $\mathfrak{D} \otimes \K$.  Hence, by Lemma~\ref{l:her}, $e \otimes e_{11} \lesssim ( p - f ) \otimes e_{11}$ in $\mathfrak{D} \otimes e_{11}$.  Thus, $e \lesssim p-f$ in $\mathfrak{D}$.  Since $p$ is norm-full in $\mathfrak{D}$ and $\mathfrak{I} \subseteq \mathfrak{D}$, by Lemma~\ref{l:fullher}, $p \mathfrak{I} p$ is norm-full in $\mathfrak{I}$. Therefore, by Lemma~\ref{lem:projfull}, $\sigma_{ \mathfrak{e} } (p)$ is norm-full in $\multialg{ \mathfrak{I} }$.  By Lemma~\ref{l:rr0full}, $\mathfrak{e}$ is a full extension.
\end{proof}

The following corollary shows that if $\mathfrak{A}$ has stable rank one, then fullness is completely determined by $K_{0} ( \mathfrak{A} )_{+}$.

\begin{corol}\label{cor:str1}
Let $\mathfrak{A}$ be a separable $C^{*}$-algebra with real rank zero and stable rank one.  Let $\mathfrak{I}$ be an ideal of $\mathfrak{A}$ such that $\mathfrak{I}$ has the corona factorization property and $\mathfrak{I}$ is stable.   Consider the extension
\begin{equation*}
\mathfrak{e} : 0 \to \mathfrak{I} \to \mathfrak{A} \to \mathfrak{A} / \mathfrak{I}  \to 0.
\end{equation*}
Then $\mathfrak{e}$ is full if and only if $\mathfrak{e}$ is stenotic and
\begin{equation*}
0 \to \mathfrak{I} \to \mathfrak{A} \to \mathfrak{A} / \mathfrak{I} \to 0 
\end{equation*}
is $K$-lexicographic.
\end{corol}

\begin{proof}
Note that every ideal and quotient of $\mathfrak{A}$ are $C^{*}$-algebras with stable rank one, which implies every ideal and quotient of $\mathfrak{A}$ are $C^{*}$-algebras with cancellation of projections. Thus, every sub-quotient of $\mathfrak{A}$ has the stable weak cancellation property.  The ``only if'' direction follows from Theorem~\ref{t:full}.  We now prove the ``if'' direction.  We must show that for every ideal $\mathfrak{D}$ of $\mathfrak{A}$ with $\mathfrak{I} \subsetneq \mathfrak{D}$, the extension
\begin{equation*}
0 \to \mathfrak{I} \to \mathfrak{D} \to \mathfrak{D} / \mathfrak{I} \to 0 
\end{equation*}
is $K$-lexicographic.  

Since 
\begin{equation*}
0 \to \mathfrak{I} \to \mathfrak{A} \to \mathfrak{A} / \mathfrak{I} \to 0 
\end{equation*}
is $K$-lexicographic and $K_{0} ( \mathfrak{A} / \mathfrak{I} )_{+} \neq K_{0} ( \mathfrak{A} / \mathfrak{I} )$, we have that 
\begin{align*}
K_{0} ( \mathfrak{A} )_{+} =   K_{0}(\iota_{ \mathfrak{I}, \mathfrak{A} } ) ( K_{0} ( \mathfrak{I} )_{+} )  \sqcup \setof{ x \in K_{0} ( \mathfrak{A} ) }{ K_{0} ( \pi_{ \mathfrak{I} , \mathfrak{A} } ) ( x ) > 0 }.
\end{align*}  
Let $\mathfrak{D}$ be an ideal of $\mathfrak{A}$ such that $\mathfrak{I} \subsetneq \mathfrak{D}$.  Since $\mathfrak{D} / \mathfrak{I}$ has stable rank one, $K_{0} ( \mathfrak{D} / \mathfrak{I} ) \neq K_{0} ( \mathfrak{D} / \mathfrak{I} )_{+}$.  Therefore, we must show that 
\begin{align*}
K_{0} ( \mathfrak{D} )_{+} =   K_{0}(\iota_{ \mathfrak{I}, \mathfrak{D} } ) ( K_{0} ( \mathfrak{I} )_{+} )  \sqcup \setof{ x \in K_{0} ( \mathfrak{D} ) }{ K_{0} ( \pi_{ \mathfrak{I} , \mathfrak{D} } ) ( x ) > 0 }.
\end{align*}
Since $\mathfrak{A}$ has stable rank one, by Lemma~\ref{l:kthyrr0},
\begin{equation*}
K_{0} ( \mathfrak{D} )_{+} \subseteq   K_{0}(\iota_{ \mathfrak{I}, \mathfrak{D} } ) ( K_{0} ( \mathfrak{I} )_{+} )  \sqcup \setof{ x \in K_{0} ( \mathfrak{D} ) }{ K_{0} ( \pi_{ \mathfrak{I} , \mathfrak{D} } ) ( x ) > 0 }.
\end{equation*}
It is clear that $K_{0}(\iota_{ \mathfrak{I}, \mathfrak{D} } ) ( K_{0} ( \mathfrak{I} )_{+} ) \subseteq K_{0} ( \mathfrak{D} )_{+}$.  Let $x \in K_{0} ( \mathfrak{D} )$ such that $K_{0} ( \pi_{\mathfrak{I}, \mathfrak{D} } )( x )  > 0$.  By Lemma~\ref{l:str1}, $K_{0} ( \pi_{ \mathfrak{I} , \mathfrak{A} } ) \circ K_{0} ( \iota_{ \mathfrak{D} , \mathfrak{A} } )(x) > 0$. Thus, $K_{0} ( \iota_{ \mathfrak{D} , \mathfrak{A} } )( x ) \in K_{0} ( \mathfrak{A} )_{+}$.  Hence, by Lemma~\ref{l:str1}, $x \in K_{0} ( \mathfrak{D} )_{ + }$.  Therefore, 
\begin{equation*}
0 \to \mathfrak{I} \to \mathfrak{D} \to \mathfrak{D} / \mathfrak{I} \to 0 
\end{equation*}
is $K$-lexicographic.  By Theorem~\ref{t:full}, $\mathfrak{e}$ is a full extension.
\end{proof}

We now show, under some hypotheses, that if $\mathfrak{A} / \mathfrak{I}$ is a purely infinite simple $C^{*}$-algebra, then $\mathfrak{e}$ is full precisely when $\mathfrak{e}$ is stenotic and $K_{0} ( \mathfrak{A} )_{+} = K_{0} ( \mathfrak{A} )$.  We first need the following proposition which shows that $\mathfrak{A}$ has a norm-full properly infinite projection  precisely when $\mathfrak{A}$ has a norm-full projection and $K_{0} ( \mathfrak{A} ) = K_{0} ( \mathfrak{A} )_{+}$. 

\begin{propo}\label{p:properlyinf}
Let $\mathfrak{A}$ be a $C^{*}$-algebra with real rank zero and has the stable weak cancellation property.  Then the follow are equivalent.
\begin{itemize}
\item[(1)] Every norm-full projection in $\mathfrak{A}$ is properly infinite.  

\item[(2)] $\mathfrak{A}$ has a norm-full properly infinite projection.  

\item[(3)] $\mathfrak{A}$ has a norm-full projection and $K_{0} ( \mathfrak{A} ) = K_{0} ( \mathfrak{A} )_{+}$.
\end{itemize}
\end{propo}

\begin{proof}
(1) implies (2) is clear.  Suppose $\mathfrak{A}$ has a norm-full properly infinite projection.  Then by Proposition~\ref{p:cuntz}, $K_{0} ( \mathfrak{A} ) = K_{0} ( \mathfrak{A} )_{++}$ which implies that $K_{0} ( \mathfrak{A} ) = K_{0} ( \mathfrak{A} )_{+}$.  Thus, (2) implies (3).

Suppose $\mathfrak{A}$ has a norm-full projection and $K_{0} ( \mathfrak{A} ) = K_{0} ( \mathfrak{A} )_{+}$.  Let $p$ be a norm-full projection of $\mathfrak{A}$.  Set $p_{1} = p \otimes e_{11}$, $p_{2} = p \otimes e_{22}$, and $p_{3} = p \otimes e_{33}$.  Then $p_{i}$ is a norm-full projection in $\mathfrak{A} \otimes \K$ and $p_{i} \sim p_{j}$.  Since $K_{0} ( \mathfrak{A} ) = K_{0} ( \mathfrak{A} )_{+}$, there exists a projection $e \in \mathfrak{A} \otimes \K$ such that $[ p _{1} ] - [ p_{2} + p_{3} ] = [ e ]$.  Choose orthogonal projections $q_{1}$ and $q_{2}$ in $\mathfrak{A} \otimes \K$ such that $q_{1} \sim p_{2} + p_{3}$ and $q_{2} \sim e$.  Then $q_{1} + q_{2}$ is a projection in $\mathfrak{A} \otimes \K$ with $[ p_{1} ] = [ q_{1} + q_{2} ]$.  Since $p_{2} + p_{3}$ is a norm-full projection in $\mathfrak{A} \otimes \K$, we have that $q_{1}$ is a norm-full projection in $\mathfrak{A} \otimes \K$.  Hence, $q_{1} + q_{2}$ is a norm-full projection in $\mathfrak{A} \otimes \K$.  Since $[ p_{1} ] = [ q_{1} + q_{2} ]$ in $K_{0} ( \mathfrak{A} )$ and since $\mathfrak{A}$ has the stable weak cancellation property, we have that $p_{1} \sim q_{1} + q_{2}$ in $\mathfrak{A} \otimes \K$.  Thus, $p_{2} + p_{3} \sim q_{1} \lesssim p_{1}$.  We have just shown that $p \oplus p \lesssim p$.  Thus, $p$ is properly infinite.  So, $p$ is a norm-full properly infinite projection in $\mathfrak{A}$.  We have just shown that (3) implies (1).
\end{proof}

\begin{corol}\label{c:fullpurelyquot}
Let $\mathfrak{A}$ be a separable $C^{*}$-algebra with real rank zero.  Let $\mathfrak{I}$ be an ideal of $\mathfrak{A}$ such that $\mathfrak{I}$ satisfies the corona factorization property, $\mathfrak{I}$ is stable, $\mathfrak{I}$ has stable weak cancellation, and $\mathfrak{A} / \mathfrak{I}$ is a purely infinite simple $C^{*}$-algebra.  Then 
\begin{align*}
\mathfrak{e} : 0 \to \mathfrak{I} \to \mathfrak{A} \to \mathfrak{A} / \mathfrak{I} \to 0
\end{align*}
is a full extension if and only if $\mathfrak{e}$ is stenotic and $K_{0} ( \mathfrak{A} ) = K_{0} ( \mathfrak{A} )_{+}$.
\end{corol}

\begin{proof}
Since $\mathfrak{A} / \mathfrak{I}$ is a purely infinite simple $C^{*}$-algebra, we have that $\mathfrak{A} / \mathfrak{I}$ has a norm-full properly infinite projection.  Suppose $\mathfrak{e}$ is a full extension.  Then by Proposition~\ref{p:sepid}, $\mathfrak{I}$ is the largest proper ideal of $\mathfrak{A}$.  Thus, the only ideal of $\mathfrak{A}$ that properly contains $\mathfrak{I}$ is $\mathfrak{A}$.   The ``only if'' direction now follows from Proposition~\ref{p:fullkthy}.  

We now prove the ``if'' direction.  First note that since $\mathfrak{A} / \mathfrak{I}$ is a purely infinite simple $C^{*}$-algebra, $\mathfrak{A} / \mathfrak{I}$ has stable weak cancellation.  Since $\mathfrak{A}$ has real rank zero, by Lemma~\ref{l:swkc}, $\mathfrak{A}$ has stable weak cancellation. 

Now observe that since $\mathfrak{e}$ is stenotic and since $\mathfrak{A} / \mathfrak{I}$ is simple, $\mathfrak{I}$ is the largest non-trivial ideal of $\mathfrak{A}$.  To use Theorem~\ref{t:full} to show that $\mathfrak{e}$ is full, we only need to show that $\mathfrak{e}$ is $K$-lexicographic.  Since $\mathfrak{A} / \mathfrak{I}$ is purely infinite simple $C^{*}$-algebra, by Proposition~\ref{p:cuntz}, $K_{0} ( \mathfrak{A} / \mathfrak{I} ) = K_{0} ( \mathfrak{A} / \mathfrak{I} )_{+} = K_{0} ( \mathfrak{A} / \mathfrak{I} ) _{++}$.  So, to show that $\mathfrak{e}$ is $K$-lexicographic, we must show that $K_{0} ( \mathfrak{A} ) = K_{0} ( \mathfrak{A} )_{+} = K_{0} ( \mathfrak{A} )_{++}$.  

Let $q$ be a non-zero projection in $\mathfrak{A} / \mathfrak{I}$.  Since $\mathfrak{A} / \mathfrak{I}$ is a simple $C^{*}$-algebra, $q$ is a norm-full projection in $\mathfrak{A} / \mathfrak{I}$.  Since $\mathfrak{A}$ has real rank zero, there exists a projection $p$ in $\mathfrak{A}$ such that $\pi_{ \mathfrak{I} , \mathfrak{A} } ( p ) = q$.  Since $\mathfrak{e}$ is stenotic and $\pi_{ \mathfrak{I} , \mathfrak{A} } (p) = q$ is norm-full in $\mathfrak{A} / \mathfrak{I}$, $p$ is norm-full in $\mathfrak{A}$.  Since $K_{0} ( \mathfrak{A} )_{+} = K_{0} ( \mathfrak{A} )$, by Proposition~\ref{p:properlyinf}, $\mathfrak{A}$ has a norm-full properly infinite projection.  Thus, by Proposition~\ref{p:cuntz}, $K_{0} ( \mathfrak{A} ) = K_{0} ( \mathfrak{A} )_{+} = K_{0} ( \mathfrak{A} )_{++}$.     
\end{proof}

Note that if the ideal lattice of $\mathfrak{A}$ is linear, then for every ideal $\mathfrak{I}$ of $\mathfrak{A}$, the extension
\begin{align*}
0 \to \mathfrak{I} \to \mathfrak{A} \to \mathfrak{A} / \mathfrak{I} \to 0
\end{align*}
is stenotic.

\begin{corol}\label{lincase}
Let $\mathfrak{A}$ be a separable $C^{*}$-algebra with real rank zero and a linear ideal lattice:
\begin{equation*}
0 = \mathfrak{I}_{0} \unlhd \mathfrak{I}_{1} \unlhd \mathfrak{I}_{2} \unlhd \cdots \unlhd \mathfrak{I}_{n-1} \unlhd \mathfrak{I}_{n} = \mathfrak{A}. 
\end{equation*}
Consider the extension 
\begin{equation*}
\mathfrak{e} : 0 \to \mathfrak{I}_{k}  \to \mathfrak{A}  \to \mathfrak{A} / \mathfrak{I}_{k}  \to 0
\end{equation*}
for $1 \leq k < n$.  Suppose $\mathfrak{I}_{k}$ is stable and either $\mathfrak{I}_{\ell+1} / \mathfrak{I}_{\ell}$ is an AF algebra or $\mathfrak{I}_{\ell+1} / \mathfrak{I}_{\ell}$ is purely infinite for each $0 \leq \ell < n$.  Then $\mathfrak{e}$ is a full extension if and only if 
\begin{align*}
0 \to \mathfrak{I}_{k} \to \mathfrak{I}_{k+1} \to \mathfrak{I}_{k+1} / \mathfrak{I}_{k} \to 0
\end{align*}
is $K$-lexicographic.
\end{corol}

\begin{proof}
Using Lemma~\ref{l:swkc}, one can show that $\mathfrak{I}_{\ell}$ has stable weak cancellation for all $\ell$.   Also, by the results in \cite{segrer:ccfis}, $\mathfrak{e}$ is a full extension if and only if
\begin{equation*}
0 \to \mathfrak{I}_{k} \to \mathfrak{I}_{k+1} \to \mathfrak{I}_{k+1} / \mathfrak{I}_{k}  \to 0
\end{equation*}
is a full extension.  The corollary now follows from Theorem~\ref{t:full} and Corollary~\ref{c:fullpurelyquot}.
\end{proof}

The authors in \cite{ERRshift} proved that if $\mathfrak{I}$ is a stable ideal of $\mathfrak{A}$ and 
\begin{align*}
\mathfrak{e} : 0 \to \mathfrak{I} \to \mathfrak{A} \to \mathfrak{A} / \mathfrak{I} \to 0
\end{align*}
is a full extension, then 
\begin{equation*}
\mathfrak{e}^{s} : 0 \to \mathfrak{I} \otimes \K \to \mathfrak{A} \otimes \K  \to ( \mathfrak{A} / \mathfrak{I} ) \otimes \K  \to 0
\end{equation*}
is a full extension.  Using our concrete $K$-theoretical test for fullness, we provide a partial converse.  This will be used in the next section to give a complete classification of non-unital graph $C^{*}$-algebras with exactly one non-trivial ideal. 

\begin{corol}\label{c:fullstable}
Let $\mathfrak{A}$ be a separable $C^{*}$-algebra with real rank zero.  Let $\mathfrak{I}$ be an ideal of $\mathfrak{A}$ such that $\mathfrak{I}$ satisfies the corona factorization property and $\mathfrak{I}$ is stable.  Suppose $\mathfrak{A}$ has stable weak cancellation and for every ideal $\mathfrak{D}$ of $\mathfrak{A}$ with $\mathfrak{I} \subsetneq \mathfrak{D}$, $\mathfrak{D} / \mathfrak{I}$ has cancellation of projections or $\mathfrak{D} / \mathfrak{I}$ has a norm-full properly infinite projection.  Then
\begin{equation*}
\mathfrak{e} : 0 \to \mathfrak{I} \to \mathfrak{A} \to \mathfrak{A} / \mathfrak{I}  \to 0
\end{equation*}
is a full extension if and only if
\begin{equation*}
\mathfrak{e}^{s} : 0 \to \mathfrak{I} \otimes \K \to \mathfrak{A} \otimes \K  \to ( \mathfrak{A} / \mathfrak{I} ) \otimes \K  \to 0
\end{equation*}
is a full extension.
\end{corol}

\begin{proof}
We have seen that $\mathfrak{e}$ is stenotic if and only if $\mathfrak{e}^{s}$ is stenotic.  Since every ideal of $\mathfrak{A} \otimes \K$ that properly contains $\mathfrak{I} \otimes \K$ is of the form $\mathfrak{D} \otimes \K$ where $\mathfrak{D}$ is an ideal of $\mathfrak{A}$ with $\mathfrak{I} \subsetneq \mathfrak{D}$, for every ideal $\mathfrak{D}_{1}$ of $\mathfrak{A} \otimes \K$ with $\mathfrak{I} \otimes \K \subsetneq \mathfrak{D}_{1}$, $\mathfrak{D}_{1} / (\mathfrak{I}  \otimes \K ) \cong ( \mathfrak{D} / \mathfrak{I} ) \otimes \K$ for some ideal $\mathfrak{D}$ of $\mathfrak{A}$ with $\mathfrak{I} \subsetneq \mathfrak{D}$.  Hence, for every ideal $\mathfrak{D}_{1}$ of $\mathfrak{A} \otimes \K$ with $\mathfrak{I} \otimes \K \subsetneq \mathfrak{D}_{1}$, $\mathfrak{D}_{1} / (\mathfrak{I}  \otimes \K )$ has cancellation of projections or has a norm-full properly infinite projection.  

Note that if $\mathfrak{D}$ is an ideal of $\mathfrak{A}$ such that $\mathfrak{I} \subsetneq \mathfrak{D}$, then
\begin{equation*}
0 \to \mathfrak{I} \to \mathfrak{D} \to \mathfrak{D} / \mathfrak{I} \to 0 
\end{equation*}
is $K$-lexicographic if and only if 
\begin{equation*}
0 \to \mathfrak{I} \otimes \K \to \mathfrak{D} \otimes \K \to (\mathfrak{D} / \mathfrak{I}) \otimes \K \to 0 
\end{equation*}
is $K$-lexicographic.  

From the above observations, the result now follows from Theorem~\ref{t:full}.
\end{proof}

\begin{remar}
The assumption that $\mathfrak{I}$ is stable in the above corollary is necessary, as we will see in Example~\ref{examplegraph}.
\end{remar}

\section{Applications}

Let $\mathfrak{A}$ be a $C^{*}$-algebra and $\mathfrak{I}$ be an ideal of $\mathfrak{A}$.  Let $\ksix ( \mathfrak{A} ; \mathfrak{I} )$ denote the six-term exact sequence in $K$-theory
\begin{align*}
\xymatrix{K_{0} ( \mathfrak{I} ) \ar[r] & K_{0} ( \mathfrak{A} ) \ar[r] & K_{0} ( \mathfrak{A} / \mathfrak{I} ) \ar[d] \\
K_{1} ( \mathfrak{A} / \mathfrak{I} ) \ar[u] & K_{1} ( \mathfrak{A} ) \ar[l] & K_{1} ( \mathfrak{I} ) \ar[l] 
}
\end{align*}
induced by the extension
\begin{align*}
0 \to \mathfrak{I} \to \mathfrak{A} \to \mathfrak{A} / \mathfrak{I} \to 0.
\end{align*}

For a $C^{*}$-algebra $\mathfrak{B}$ with ideal $\mathfrak{D}$, a homomorphism $\ftn{ ( \beta_{*}, \eta_{*} , \alpha_{*} ) }{ \ksix ( \mathfrak{A} ; \mathfrak{I} ) }{ \ksix ( \mathfrak{B} ; \mathfrak{D} ) }$ is six group homomorphisms $\ftn{ \beta_{i} }{ K_{i} ( \mathfrak{I} ) }{K_{i} ( \mathfrak{D} ) }$, $\ftn{ \eta_{i} }{ K_{i} ( \mathfrak{A} ) }{ K_{i} ( \mathfrak{B} ) }$, and $\ftn{ \beta_{i} }{ K_{i} ( \mathfrak{A} / \mathfrak{I} ) }{ K_{i} ( \mathfrak{B} / \mathfrak{D} ) }$ for $i = 0, 1$, making the obvious diagrams commute.  An isomorphism $\ftn{ ( \beta_{*}, \eta_{*} , \alpha_{*} ) }{ \ksix ( \mathfrak{A} ; \mathfrak{I} ) }{ \ksix ( \mathfrak{B} ; \mathfrak{D} ) }$ is defined in the obvious way.

\subsection{A class of examples}

The following class of examples quite efficiently illustrates a number of key points in this section. Consider the class of graphs $G[m,(n_i)]$ defined by the adjacency matrix
\[
\begin{bmatrix}
m&n_1&n_2&n_3&n_4&\dots\\
 &   &    2&    &&\\
& &   &    2&    &&\\
&&&&\ddots&&
\end{bmatrix}
\]
or graphically as
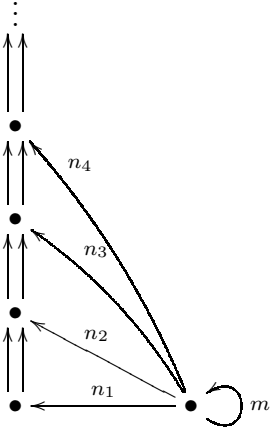
\[
\xymatrix{
{\vdots}&&\\
{\bullet}\ar@<-1mm>[u]\ar@<+1mm>[u]&&\\
{\bullet}\ar@<-1mm>[u]\ar@<+1mm>[u]&&\\
{\bullet}\ar@<-1mm>[u]\ar@<+1mm>[u]&&\\
{\bullet}\ar@<-1mm>[u]\ar@<+1mm>[u]&&{\bullet}\ar@(rd,ru)[]_-m
\ar[ll]_-{n_1}\ar[ull]_(0.6){n_2}\ar@/_/[uull]_(0.7){n_3}\ar@/_/[uuull]_(0.8){n_4}}
\]

We want to study $C^*(G[m,(n_i)])$ under the added assumption that this $C^*$-algebra has precisely one ideal $\mathfrak I$. Then at least one $n_i$ must be nonzero, $m\not=1$ must hold to ensure condition (K), and when $\infty> m>1$ we must have that $\sum_{i=1}^\infty n_i<\infty$.

We have that the $K_1$-groups of $\mathfrak I, C^*(G[m,(n_i)])$ and $C^*(G[m,(n_i)])/\mathfrak I$ vanish, and
by standard methods (cf.\ \cite{tmcsemt:imkga}), we may compute
\[
\xymatrix{
{0}\ar[r]&K_0(\mathfrak I)\ar[r]&K_0(C^*(G[m,(n_i)]))\ar[r]&K_0(C^*(G[m,(n_i)])/\mathfrak I)\ar[r]&0.}
\]

When $m=0$ or $m=\infty$, we get that the sequence splits and takes the form
\[
\xymatrix{
0\ar[r]&{\ZZ\left[\frac 12\right]}\ar[r]&{\ZZ\left[\frac 12\right]\oplus \ZZ}\ar[r]&\ZZ\ar[r]&0}.
\]
In the case $m=0$, the middle group is ordered by 
\[
\bigcup_{ n = 0}^{ \infty} \left[  \left( (-n \alpha , \infty ) \cap \ZZ  \left[ \frac{1}{2} \right] \right) \times \{n\} \right]
\]
where $\alpha = \sum_{i=1}^\infty n_i2^{-i}$, and the other two groups are ordered canonically as subsets of $\RR$. When $m=\infty$ the middle group is trivially ordered; every element is positive. 

In the remaining cases $\infty>m>1$, again, the middle group is trivially ordered and the sequence then takes the form
\[
\xymatrix{
0\ar[r]&{\ZZ\left[\frac 12\right]}\ar[r]&{\ZZ\left[\frac 12\right]\oplus \ZZ/x\ZZ}\ar[r]^-{\pi}&\ZZ/(m-1)\ZZ\ar[r]&0}.
\]
Here, one may prove that $x$ is an integer depending on the number
\[
N=\sum_{i=1}^k n_i2^{k-i}
\]
where $k$ is chosen such that $n_i=0$ for $i>k$,
and ranges among all numbers of the form
\[
2^\ell p_1^{i_1}\cdots p_r^{i_r}, 0\leq i_j\leq k_j
\]
where $m-1=2^\ell p_1^{k_1}\cdots p_r^{k_r}$.

The reader is asked to note that in the first (AF) case, the extension is always the same but may be ordered in uncountably many ways whereas in the second (mixed) case, the extensions may vary but are always ordered similarly. We will elaborate on this in the next section.

We end this section by showing that the assumption that $\mathfrak{I}$ is a stable $C^{*}$-algebra in Corollary~\ref{c:fullstable} is necessary.

\begin{examp}\label{examplegraph}
Let $\{ n_{i} \}_{ i = 1}^{ \infty }$ be a sequence such that $\alpha = \sum_{ i = 1}^{ \infty } n_{i} 2^{-i} = 1$.  Set $\mathfrak{A} = C^*(G[0,(n_i)])$.  Then $\mathfrak{A}$ is an AF algebra with exactly one ideal $\mathfrak{I}$ such that $\mathfrak{I} \otimes \K \cong \mathsf{M}_{2^{\infty}} \otimes \K$, and 
\begin{align*}
K_{0} ( \mathfrak{A} ) &= \Z\left[ \frac{1}{2} \right] \oplus \Z \\
K_{0} ( \mathfrak{A} )_{+} &= \bigcup_{ n = 0}^{ \infty} \left[  \left( (-n  , \infty ) \cap \ZZ  \left[ \frac{1}{2} \right] \right) \times \{n\} \right].
\end{align*}
Thus, 
\begin{align*}
K_{0} ( \mathfrak{A} \otimes \K ) &= \Z\left[ \frac{1}{2} \right] \oplus \Z \\
K_{0} ( \mathfrak{A} \otimes \K )_{+} &= \bigcup_{ n = 0}^{ \infty} \left[  \left( (-n  , \infty ) \cap \ZZ  \left[ \frac{1}{2} \right] \right) \times \{n\} \right].
\end{align*}
Hence, by Corollary~\ref{cor:str1}, 
\begin{align*}
0 \to \mathfrak{I} \otimes \K \to \mathfrak{A} \otimes \K \to ( \mathfrak{A} / \mathfrak{I} ) \otimes \K \to 0 
\end{align*}
is not a full extension.

A computation shows that $\mathfrak{I}$ is a non-unital, non-stable $C^{*}$-algebra.  Therefore, $\mathfrak{I} = \overline{ a ( \mathsf{M}_{ 2^{\infty }  }\otimes \K ) a }$ is a $C^{*}$-algebra with continuous scale.  By Theorem~2.8 of \cite{contscale}, $\corona{ \mathfrak{I} }$ is a simple $C^{*}$-algebra.  Since 
\begin{align*}
\mathfrak{e} : 0 \to \mathfrak{I} \to \mathfrak{A} \to \mathfrak{A} / \mathfrak{I} \to 0
\end{align*}
is an essential extension and since $\corona{ \mathfrak{I} }$ is a simple $C^{*}$-algebra, we have that $\mathfrak{e}$ is a full extension.  By the above paragraph, $\mathfrak{e}^{s}$ is not a full extension.
\end{examp}

\subsection{Classification up to stable isomorphism}

In our first classification result, Theorem~3.9 of \cite{ERRshift} for full extensions of classifiable simple $C^*$-algebras the chosen invariant was the six-term exact sequence $\ksix(-;-)$ where the $K_0$-groups of the ideal and quotient were to be considered as ordered groups. 
Thus, although also the $K_0$-group of the middle $C^*$-algebra carries a natural ordering, it turns out to be a disposable part of the invariant. The same phenomenon occurs in all the ensuing classification results in \cite{semt_classgraphalg} and \cite{segrer:ccfis} which are --- directly or indirectly --- based on the same approach.

The results above explain the phenomenon in many cases, by demonstrating that in a full extension, the order of the $K_0$-group in the middle is ordered in a way which is determined by the order of the other two. The phenomenon manifests  most clearly for the graph $C^*$-algebras $C^*(E)$ with one ideal classified up to stable isomorphism in \cite{semt_classgraphalg}. Using notation introduced there, these fall in four classes according to whether the (simple) ideal and quotient are purely infinite or AF: 
\begin{center}
\begin{tabular}{|c|c|c|}\hline Case&$\mathfrak I$&$\mathfrak A/\mathfrak I$\\ \hline $\fifi$&AF&AF\\
$\fiin$&AF&PI\\ $\infi$&PI&AF\\
$\inin$&PI&PI\\ \hline
\end{tabular}
\end{center} 
and in all cases except $\fifi$ it is proved in \cite{semt_classgraphalg} that the extensions are automatically full. Hence the order of $K_0(C^*(E))$ is determined by our results above. We note

\begin{propo}\label{grint}
Let $C^*(E)$ be a graph $C^*$-algebra with precisely one ideal $\mathfrak I$.
\begin{enumerate}[(i)]
\item In cases $\fiin$ and $\inin$, we have
\[
K_0(C^*(E))_+=K_0(C^*(E)),
\]
\item in case $\infi$, we have
\[
K_0(C^*(E))_+=K_{0} ( \iota ) (K_0(\mathfrak I)_+)\sqcup \{x\in K_0(C^*(E))\mid K_{0}(\pi)(x)>0\}
\]
\end{enumerate}
and hence in these cases, the order on $K_0(C^*(E))$ is determined by the order on $K_0(\mathfrak I)$ and $K_0(C^*(E)/\mathfrak I)$. But 
in case $\fifi$, $K_0(C^*(E))_+$  is not determined in this way. In fact,
\begin{enumerate}[(i)]\addtocounter{enumi}{2}
\item  there exists a family of graphs $E_\lambda$ such that all invariants $\ksix(C^*(E_\lambda);\mathfrak I_\lambda)$ 
are the same except for the ordering on $K_0(C^*(E_\lambda))$, all of which  are mutually non-isomorphic.
\end{enumerate}
\end{propo}
\begin{proof}
Since the extensions are seen to be full in the course of the proof of Theorem~4.5 in \cite{semt_classgraphalg}, and since the $C^*$-algebras in question have stable weak cancellation according to \cite{amp:nonstablekthy} and Lemma~\ref{l:swkc}, we may apply Theorem~\ref{t:full} to prove claims (i) and (ii).

For claim (iii), consider the class of graphs $G[0,(n_i)]$ considered above. With $\alpha=\sum_{i=1}^\infty n_i2^{-i}$ the order on $K_0(C^*(G[0,(n_i)]))=\ZZ[\frac12]\oplus\ZZ$ is, as we have seen, given by
\[
\bigcup_{ n = 0}^{ \infty} \left[  \left( (-n \alpha , \infty ) \cap \ZZ  \left[ \frac{1}{2} \right] \right) \times \{n\} \right]
\]
and hence different for different $\alpha$. Since there are only countably many automorphisms of $\ZZ[\frac 12]\oplus \ZZ$, the claim follows.  
\end{proof}

Note that among the graph $C^*$-algebras associated to $G[0,(n_i)]$, the ones for which the corresponding extensions are full are precisely those with $\alpha=\infty$.
In fact it is the rule rather than the exception that there are both full and non-full extensions when the ideal is stably finite. For instance, we note:

\begin{examp}\label{nonfullext}
Let $\mathfrak{B}$ be the stabilization of the UHF-algebra of type $2^{\infty}$.  Then there exist essential extensions
\begin{equation*}
\mathfrak{e}_{i}:  0 \to \mathfrak{B} \to \mathfrak{E}_{i} \to \C \to 0
\end{equation*}
for $i=1,2$ such that $\mathfrak{e}_{1}$ is full but $\mathfrak{e}_{2}$ is not. The real rank of $\mathfrak E_i$ is zero, and the invariants $\ksix(\mathfrak E_i;\mathfrak B)$ are the same except for the ordering of $K_0(\mathfrak E_i)$.
\end{examp}
\begin{proof}
By \cite{hl:imasac}, we can choose a non-full non-zero projection $p\in \multialg{\mathfrak B}$. Let $\gamma_{1}$ be the composition
\begin{align*}
\C \overset{ \mathrm{unital} }{ \hookrightarrow } p \multialg{ \mathfrak{B} } p \hookrightarrow  \multialg{ \mathfrak{B} }.
\end{align*}    
Let $\ftn{ \gamma_{2} }{ \C }{ \multialg{ \mathfrak{B} } }$ be any full homomorphism.  Let $\mathfrak{E}_{1} =  \C \oplus_{ \pi_{ \mathfrak{B}  } \circ \gamma_{1} , \pi_{ \mathfrak{B} } } \multialg{ \mathfrak{B} }$ and $\mathfrak{E}_{2} = \C \oplus_{ \pi_{\mathfrak{B} } \circ \gamma_{2} , \pi_{ \mathfrak{B} } } \multialg{ \mathfrak{B} }$.  
\end{proof}

\begin{examp}\label{nonfullextagain}
Let $\mathfrak{B}_{0}$ be a unital simple $AF$ algebra with finitely many extremal traces such that $\mathfrak{B} = \mathfrak{B}_{0} \otimes \KKK$ is not isomorphic to $\KKK$.  Then there exist an extension
\begin{equation*}
\mathfrak{e}_{i}:  0 \to \mathfrak{B} \to \mathfrak{E}_{i} \to \mathcal{O}_2 \to 0
\end{equation*}
for $i=1,2$ such that $\mathfrak{e}_{1}$ is full but $\mathfrak{e}_{2}$ is not. The real rank of $\mathfrak E_i$ is zero, and the invariants $\ksix(\mathfrak E_i;\mathfrak B)$ are the same except for the ordering of $K_0(\mathfrak E_i)$. 
\end{examp}
\begin{proof}
By Theorem~1.3 of \cite{SZ_proj}, $\corona{\mathfrak{B} }$ is purely infinite.  Let $\ftn{ \tau_{1} }{ \mathcal{O}_{2} }{ \corona{ \mathfrak{B} } }$ be any full homomorphism which always exists by \cite{ekncp:eeccao}.  By \cite{mr_ideal}, $\corona{ \mathfrak{B} }$ has a smallest non-trivial ideal, $\mathfrak{I}$.  Since $\corona{ \mathfrak{B} }$ is purely infinite and $\mathfrak{I}$ is simple, there exists an injective homomorphism $\ftn{ \tau_{2} }{ \mathcal{O}_{2} }{  \mathfrak{I} \subsetneq \corona{\mathfrak{B} } }$.  Let $\mathfrak{e}_{i}$ be the extension associated to the Busby invariant $\tau_{i}$.
\end{proof}


We end this section by showing that even if one is prepared to use the order on $K_{0} ( \mathfrak{C} )$, the invariants are not complete for non-full essential extensions of classifiable simple $C^{*}$-algebras.  We first need the following result.  It is proved in the exact same way as in the proof of Proposition~2.2 (c) of \cite{SZ_proj}.

\begin{lemma}\label{l:equivalentproj}
Let $\mathfrak{B}_{0}$ be a unital, simple, AF-algebra with finitely many extreme tracial states $\{ \tau_{1}, \dots, \tau_{n} \}$.  Set $\mathfrak{B} = \mathfrak{B}_{0} \otimes \K$.  Suppose $p \in \mathfrak{B}$ and $q \in \multialg{ \mathfrak{B} }$ are projections such that $\overline{\tau}_{i} ( p ) < \overline{\tau}_{i} ( q )$, where $\overline{ \tau }_{i}$ is the extension of $\tau_{i}$ to $\multialg{ \mathfrak{B} }$.  Then there exists a projection $e\in \mathfrak{B}_{0} \otimes \K$ such that $p \sim e$ and $e \leq q$.
\end{lemma}

\begin{propo}
There exist separable, nuclear, real rank zero $C^{*}$-algebras $\mathfrak{E}_{1}$ and $\mathfrak{E}_{2}$ satisfying the UCT such that $\mathfrak{E}_{i}$ has a unique nontrivial ideal $\mathfrak{B}_{i}$, $\mathfrak{E}_{1}$ and $\mathfrak{E}_{2}$ are not stably isomorphic but there exists an isomorphism $\ftn{ ( \beta_{*}, \eta_{*} , \alpha_{*} ) }{ \ksix ( \mathfrak{E}_{1}; \mathfrak{B}_{1} ) }{ \ksix ( \mathfrak{E}_{2} ; \mathfrak{B}_{2} ) }$ such that $\alpha_{0}$, $\beta_{0}$, and $\eta_{0}$ are positive isomorphisms.
\end{propo}

\begin{proof}
Let $\mathfrak{A} = \mathcal{O}_{2}$ and let $\mathfrak{B}_{0}$ be a unital simple AF algebra with two extreme tracial states $\{ \tau_{0}, \tau_{1} \}$ such that $K_{0} ( \mathfrak{B}_{0} ) \cong \Q \oplus \Q $ and the isomorphism takes $K_{0} ( \mathfrak{B}_{0} )_{+}$ to
\begin{align*}
G = \setof{ ( x, y ) }{ \text{$x > 0$ and $y > 0$} } \sqcup \{ (0,0) \}. 
\end{align*}  
In fact the isomorphism is given by $g \mapsto ( \tau_{0} ( g), \tau_{1} ( g ) )$.  Set $\mathfrak{B} = \mathfrak{B}_{0} \otimes \K$.

By \cite{mr_ideal}, $\corona{ \mathfrak{B} }$ has exactly three non-trivial proper ideals, namely $\mathfrak I_0=\mathfrak D_0/\mathfrak B$ and $\mathfrak I_1=\mathfrak D_1/\mathfrak B$ with $\mathfrak D_i$ the ideal generated by
\begin{align*}
\setof{ x \in \multialg{ \mathfrak{B} } }{ \text{$\overline{\tau}_{i} ( x^{*} x ) < \infty$} },
\end{align*}
for $i\in\{0,1\}$, and the nonzero minimal ideal $\mathfrak I_2=\mathfrak I_0\cap\mathfrak I_1$.

By Theorem~1.3 of \cite{SZ_proj}, $\mathfrak I_0$, $\mathfrak{I}_{1}$ and $\mathfrak{I}_{2}$ are properly infinite $C^{*}$-algebras.  Hence, for $i\in\{0,1,2\}$, there exists an injective homomorphism $\ftn{ \beta_{i} }{ \mathfrak{A} }{ \mathfrak{I}_{i} } $ such that for each nonzero $a \in \mathfrak{A}$, $\beta_{i}(a)$ is norm-full in $\mathfrak{I}_{i}$.
Set $\mathfrak{E}_{i} = \mathfrak{A} \oplus_{ \beta_{i} , \pi_{ \mathfrak{B} } } \multialg{ \mathfrak{B} }$.  It is clear that $\mathfrak{E}_{i}$ is a separable, nuclear, real rank zero $C^{*}$-algebra satisfying the UCT with $\mathfrak{B}$ as its unique non-trivial ideal.  We also have non-full extensions
\begin{align*}
\mathfrak{e}_{i}: 0 \to \mathfrak{B}  \to \mathfrak{E}_{i}  \to \mathfrak{A}  \to 0. 
\end{align*}  
Moreover, $\ksix ( \mathfrak{E}_{i} ; \mathfrak{B} )$ is isomorphic to 
\begin{align*}
\xymatrix{ K_{0} ( \mathfrak{B} ) \ar[r]^{ K_{0} ( \iota_{i} ) } & K_{0} ( \mathfrak{E}_{i} ) \ar[r] &  0 \ar[d]\\0\ar[u]&0\ar[l]&0\ar[l]}
\end{align*}
where $\iota_{i}$ is the inclusion of $\mathfrak{B}$ into $\mathfrak{E}_{i}$ and $\pi_{i}$ is the projection from $\mathfrak{E}_{i}$ to $\mathcal{O}_{2}$.  Also note that if $x \in (\mathfrak{E}_{i})_{+}$ such that $x$ is not in $\mathfrak{B}$, then $\overline{\tau}_{i} ( \sigma_{ \mathfrak{e}_{i} } (x) ) < \infty$ and $\overline{ \tau }_{j} ( \sigma_{ \mathfrak{e}_{i} } ( x ) ) = \infty$, where $j \neq i$.

\medskip

\noindent \emph{Claim 1:  If $p$ is a projection in $\mathsf{M}_{n} ( \mathfrak{B} )$, then there exists a projection $q \in \mathsf{M}_{k} ( \mathfrak{E}_{i} )$ such that $[ \iota_{i} ( p ) ] = [ q ]$ and $\pi_{i} ( q ) \neq 0$.}  Let $e$ be a projection in $\mathfrak{E}_{i}$ such that $\pi_{i} ( e )$ is non-zero.  Let $p$ be a non-zero projection in $\mathsf{M}_{n} ( \mathfrak{B} )$.  Then $[ e ] - [ \iota_{i} ( p ) ] = [ \iota_{i} ( p_{1} ) ] - [ \iota_{i} ( p_{2} ) ]$ for projections $p_{1}, p_{2} \in \mathsf{M}_{\ell} ( \mathfrak{B} )$.  Then $[ e \oplus \iota_{i} ( p_{2} ) ] - [ \iota_{i} ( p_{1} ) ] = [ \iota_{i} ( p ) ]$.  Set $q_{2} = \sigma_{ \mathfrak{e}_{i} } ( e \oplus \iota_{i} (p_{2} ) )$ and $q_{1} = \sigma_{ \mathfrak{e}_{i} } (  \iota_{i} (p_{1} ) )$.  Then, $\overline{ \tau }_{i}  ( q_{2} ) >  \overline{ \tau }_{i}  ( q_{1} )$ and $\overline{ \tau }_{j}  ( q_{2} ) = \infty >  \overline{ \tau }_{j}  ( q_{1} )$ for $j \neq i$.  Since $q_{1} \in \mathsf{M}_{\ell} (\mathfrak{B})$, by Lemma~\ref{l:equivalentproj}, there exists $e \in \mathfrak{B}$ such that $e \sim q_{1}$ and $e \leq q_{2}$.  By Lemma~\ref{l:essentialmv}, $\iota_{i} ( p_{1} )$ is Murray-von Neumann equivalent to a subprojection $f$ of $e \oplus \iota_{i} ( p_{2} )$ in $\mathsf{M}_{\ell+1} ( \mathfrak{E}_{i} )$.  Hence, $[ \iota_{i} ( p ) ] = [ e \oplus \iota_{i} ( p_{2} ) ] - [ \iota_{i} ( p_{1} ) ] = [ e \oplus \iota_{i} ( p_{2} ) - f ]$.  Set $q = e \oplus \iota_{i} ( p_{2} ) - f$.  Then $q$ is in $\mathsf{M}_{\ell+1} ( \mathfrak{E}_{i} )$ with $\pi_{i}( q ) = \pi_{i} ( e ) \neq 0$ and $[ \iota_{i} ( p ) ] = [ q ]$.  

\medskip

\noindent \emph{Claim 2:  Suppose $x = [ p ] - [ q ] \in K_{0} ( \mathfrak{E}_{i} )$ with projections $p$ and $q$ in $\mathfrak{E}_{i} \otimes \K$ such that $\pi_{i} ( p)$ and $\pi_{i} ( q )$ are non-zero projections and suppose $\overline{\tau}_{i} ( \sigma_{ \mathfrak{e}_{i} } ( p ) ) > \overline{ \tau }_{i} (  \sigma_{ \mathfrak{e}_{i} } (q) )$.  Then $x \in K_{0} ( \mathfrak{E}_{i} )_{+}$.}  Since $\mathcal{O}_{2} \otimes \K$ is a purely infinite simple $C^{*}$-algebra, there exists a projection $e$ in $\mathcal{O}_{2} \otimes \K$ such that $\pi_{i} ( q ) \sim e$ and $e < \pi_{i} ( p )$.  By Lemma~\ref{l:strict}, there exists $v \in \mathfrak{E}_{i} \otimes \K$ such that $\pi_{i} ( v^{*} v ) = \pi_{i} ( q )$, $v^{*}v \leq q$, $vv^{*} \leq p$, and $\pi_{i} ( vv^{*} ) \neq \pi_{i} ( p)$. Therefore,
\begin{align*}
\overline{\tau}_{i} ( \sigma_{ \mathfrak{e}_{i} } ( p - vv^{*} ) ) > \overline{ \tau }_{i} (  \sigma_{ \mathfrak{e}_{i} } (q - v^{*} v ) )
\end{align*}
and 
\begin{align*}
\overline{\tau}_{j} ( \sigma_{ \mathfrak{e}_{i} } ( p - vv^{*} ) ) = \infty > \overline{ \tau }_{j} (  \sigma_{ \mathfrak{e}_{i} } (q - v^{*} v ) )
\end{align*}
for $j \neq i$.  By Lemma~\ref{l:equivalentproj}, $\sigma_{ \mathfrak{e}_{i} } (q - v^{*} v ) \sim e$ where $e$ is a projection in $\mathfrak{B} \otimes \K$ with $e \leq \sigma_{ \mathfrak{e}_{i} } ( p - vv^{*} )$.  By Lemma~\ref{l:essentialmv}, $q - v^{*} v$ is Murray-von Neumman equivalent to a sub-projection of $p - vv^{*}$.  Hence, 
\begin{align*}
[ q ] = [ q - v^{*} v ] + [ v^{*} v ] \leq [ p - vv^{*} ] + [ vv^{*} ] = [ p ].
\end{align*}
Thus, $x \in K_{0} ( \mathfrak{E}_{i} )_{+}$.  

\medskip

\noindent \emph{Claim 3:  The isomorphism $\lambda_{i} : K_{0} ( \mathfrak{E}_{i} ) \cong \Q \oplus \Q $ induced by $K_{0} ( \iota_{i} )$ takes $K_{0} ( \mathfrak{E}_{i} )_{+}$ onto 
\begin{align*}
\mathcal{H}_{i} = \setof{ ( x, y ) \in \Q \oplus \Q }{ \rho_{i} ( (x,y) ) > 0 } \sqcup \{ (0,0) \}.
\end{align*}
where $\ftn{ \rho_{i} }{ \Q \oplus \Q }{ \Q }$ is the projection in the $i$th coordinate.}  Note that if $p$ and $q$ are projections in $\mathfrak{E}_{i}$ such that $[ p ] = [ q ]$ in $K_{0} ( \mathfrak{E}_{i} )$, then $\tau_{i} ( \sigma_{ \mathfrak{e}_{i} } (p) ) = \tau_{i} ( \sigma_{ \mathfrak{e}_{i} } ( q ) )$. 

We will prove Claim 3 for $i = 1$.  The case $i = 2$ is similar.  Let $a \in K_{0} (\mathfrak{E}_{1} )$ such that $\lambda_{1}(a) =(x,y) \in \Q \oplus \Q $.  Suppose $x > 0$.  

\medskip

\emph{Case 1: Suppose $y > 0$.}  Then there exists a projection $e \in \mathfrak{B} \otimes \K$ such that 
\begin{align*}
( \tau_{1} ( [e] ) , \tau_{2} ( [e] ) ) = ( x, y ).
\end{align*}  
Hence, $a = K_{0} ( [ \iota_{1}(e) ] ) \in K_{0} ( \mathfrak{E}_{1} )_{+}$.   

\medskip

\emph{Case 2:  Suppose $y \leq 0$.}  Note that there exist projections $e_{1} , e_{2} \in \mathfrak{B} \otimes \K$ such that $( \tau_{1} ( [e_{1}] ) , \tau_{2} ( [e_{1}] ) ) = ( x+1, 1 )$ and $( \tau_{1} ( [e_{2}] ) , \tau_{2} ( [e_{2}] ) ) = ( 1, 1-y )$.  Choose projections $q_{1}, q_{2} \in \mathfrak{E}_{1} \otimes \K$ such that $\pi_{1} ( q_{j} ) \neq 0$ and $[ q_{j} ] = K_{0} ( \iota_{1} ) ( [ e_{j} ] )$.  Hence, $\overline{\tau}_{1} ( \sigma_{\mathfrak{e}_{1}} ( q_{j} ) ) = \overline{\tau}_{1} (  e_{j}  )$.  Therefore,
\begin{align*}
\overline{\tau}_{1} ( \sigma_{\mathfrak{e}_{1}} ( q_{1} ) ) &= \overline{\tau}_{1} ( e_{1}  ) = \tau_{1} ( e_{1} ) = x+1 \\
									&> 1 = \tau_{1} ( e_{2} ) =  \overline{\tau}_{1} ( e_{2}  ) = \overline{\tau}_{1} ( \sigma_{\mathfrak{e}_{1}} ( q_{2} ) ).
\end{align*}
By Claim 2, $[q_{1}] - [ q_{2} ] \in K_{0} ( \mathfrak{E}_{1} )_{+}$.  Note that $\lambda_{1} ( [q_{1} ] - [ q_{2} ] ) = ( x + 1, 1 ) - ( 1 , 1 - y ) = ( x , y )$.  Hence, $a = [ q_{1} ] - [q_{2} ] \in K_{0} ( \mathfrak{E}_{1} )_{+}$.

Let $p$ be a projection in $\mathfrak{E}_{1} \otimes \K$ and $[ p ] \neq 0$.  Since $\mathfrak{E}_{1} \otimes \K$ has real rank zero, there exists an increasing sequence of projections $\{ e_{n} \}_{n \in \N }$ in $\mathfrak{B} \otimes \K$ that converges in the strict topology of $\multialg{ \mathfrak{B} \otimes \K }$ to $\sigma_{ \mathfrak{e}_{1} } ( p )$.  Then 
\begin{align*}
\overline{ \tau}_{1} ( \sigma_{ \mathfrak{e}_{1} } ( p ) ) = \sup \setof{ \tau_{1} ( e_{n} ) }{ n \in \N }. 
\end{align*}
Suppose $\overline{ \tau }_{1} ( \sigma_{ \mathfrak{e}_{1} } ( p ) ) = 0$.  Then $\tau_{1} ( e_{n} ) = 0$ for all $n \in \N$.  Therefore, $( 0 , \tau_{2} ( [e_{n}] ) ) = ( \tau_{1} ( [e_{n}] ) , \tau_{2} ( [e_{n}] ) ) \in G$ for all $n \in \N$.  Hence, $\tau_{1} ( [e_{n}] ) = 0$ and $\tau_{2} ( [ e_{n} ] ) = 0$ for all $n \in \N$.  Thus, $ [ e_{n} ] = 0$ for all $n \in \N$.  Since $\mathfrak{B} \otimes \K$ is an AF-algebra, $e_{n} = 0$ for all $n \in \N$ which would imply that $[ p ] = 0$.  Thus, $\overline{ \tau }_{1} ( \sigma_{ \mathfrak{e}_{1}} (p) ) > 0$.  

Let $q_{1}, q_{2}$ be projections in $\mathfrak{B} \otimes \K$ such that $[ p ] = [ \iota_{1} ( q_{1} ) ] - [ \iota_{2} ( q_{2} ) ]$.  Then
\begin{align*}
\tau_{1} ( q_{1} ) - \tau_{1} ( q_{2} ) = \overline{ \tau }_{1} ( \sigma_{ \mathfrak{e}_{1} } ( p ) ) > 0.
\end{align*}
Hence, $\rho_{1} ( \lambda_{1} ( [ p ] )  ) =  \tau_{1} ( [q_{1}] ) - \tau_{1} ( [q_{2}] ) = \tau_{1} ( q_{1} ) - \tau_{2} ( q_{2} ) > 0$.  We have just shown that $\lambda_{1}$ takes $K_{0} ( \mathfrak{E}_{1} )_{+}$ onto $\mathcal{H}_{1}$.

Define $\eta_{0} ( x, y ) = ( y, x )$, $\beta_{0} ( x, y) = ( y, x )$, $\beta_{1} = \eta_{1} = \alpha_{0} = \alpha_{1} = 0$.  By the above paragraphs, $\ftn{ ( \beta_{*}, \eta_{*} , \alpha_{*} ) }{ \ksix ( \mathfrak{E}_{1}; \mathfrak{B} ) }{ \ksix ( \mathfrak{E}_{2} ; \mathfrak{B} ) }$ such that $\alpha_{0}$, $\beta_{0}$, and $\eta_{0}$ are positive isomorphisms.  Suppose now that $\mathfrak{E}_{1}$ and $\mathfrak{E}_{2}$ are stably isomorphic.  By Theorem~2.2 of \cite{ELPmorph}, there exist isomorphisms $\ftn{ \phi }{ \mathfrak{A} \otimes \K }{ \mathfrak{A} \otimes \K }$ and $\ftn{ \psi }{ \mathfrak{B} \otimes \K }{ \mathfrak{B} \otimes \K }$ such that 
\begin{align}\label{commdiagm}
\vcenter{\xymatrix{
\mathfrak{A} \otimes \K \ar[r]^-{ \tau_{ \mathfrak{e}_{1}^{s} } } \ar[d]_{\phi} & \corona{ \mathfrak{B} \otimes \K } \ar[d]^{ \overline{ \psi } } \\
\mathfrak{A} \otimes \K \ar[r]_-{ \tau_{ \mathfrak{e}_{2}^{s} } } & \corona{ \mathfrak{B} \otimes \K } 
}
}
\end{align} 

Let $\mathfrak{D}_{2}'$ be the ideal of $\multialg{ \mathfrak{B} \otimes \K }$ generated by 
\begin{align*}
\setof{ x \in \multialg{ \mathfrak{B} \otimes \K } }{ \text{$\tau_{0} ( x^{*} x ) < \infty$ and $\tau_{1} ( x^{*} x ) < \infty$} },
\end{align*} 
and for $i\in\{0,1\}$ let $\mathfrak{D}_{i}'$ be the ideal of $\multialg{ \mathfrak{B} \otimes \K }$ generated by 
\begin{align*}
\setof{ x \in \multialg{ \mathfrak{B} \otimes \K } }{ \text{$\tau_{i} ( x^{*} x ) < \infty$} }.
\end{align*}
Set for all $i\in\{0,1,2\}$ $\mathfrak{I}_{i}' = \mathfrak{D}_{i}' / (\mathfrak{B} \otimes \K)$.  By \cite{mr_ideal}, $\mathfrak{I}_{0}', \mathfrak{I}_{1}'$, and $\mathfrak{I}_{2}'$ are the only non-trivial proper ideals of $\corona{ \mathfrak{B} \otimes \K }$.   Note that $\ftn{ \tau_{\mathfrak{e}_{i}^{s}}}{ \mathfrak{A} \otimes \K }{ \mathfrak{I}_{i}' }$ and $\tau_{\mathfrak{e}_{i}}^{s} ( a )$ is norm-full in $\mathfrak{I}_{i}'$.  By (\ref{commdiagm}), we have that for each $a \in \mathfrak{A} \otimes \K$, $\tau_{ \mathfrak{e}_2 }^s(a)$ is norm-full in $\mathfrak{I}_1$ which is a contradiction to the fact that the image of $\tau_{ \mathfrak{e}_2 }^s$ is in $\mathfrak{I}_2' \subsetneq \mathfrak{I}_1'$.
\end{proof}

\subsection{Range and permanence}
 
It turns out that for graph $C^*$-algebras $C^*(E)$ with precisely one ideal, as well as in some other classifiable classes, the condition on the order of $K_0(C^*(E))$ which interprets the necessary presence of fullness in cases $\fiin$, $\infi$ and $\inin$ combines with the similary necessary $K$-theoretical interpretation of real rank zero to form an \textbf{sufficient} condition of when a given six-term exact sequence
\[
\xymatrix{
{G'}\ar[r]&{G}\ar[r]&{G''}\ar[d]^{\partial_0}\\
{H''}\ar[u]&{H}\ar[l]&{H'}\ar[l]}
\]
is the $K$-theory of a stable graph $C^*$-algebra with precisely one ideal, provided one knows already that $G'\oplus H'$ and $G''\oplus H''$ are the (ordered) $K$-groups of simple and stable graph $C^*$-algebras.

In \cite{ektw} we shall prove that this will be the case precisely when $\partial_0=0$ and $G$ is ordered in a way which is consistent with Proposition~\ref{grint} above, the key property being
\[
(G'')_+=G''\Longrightarrow G_+=G
\]
(cf. Case (2) of $K$-lexicographic in Definition~\ref{d:klexi}).  Combining this \textbf{range result} with the stable classification result obtained in \cite{semt_classgraphalg} we arrive at a similar permanence result, giving a complete $K$-theoretical description of when $\mathfrak E$ fitting in
\[
  0 \to C^*(E) \to \mathfrak{E} \to C^*(F) \to 0
\]
is a (stable) graph $C^*$-algebra, provided that $C^*(E)$ and $C^*(F)$ are stable and simple.

\subsection{Exact classification}

The classification result for graph $C^*$-algebras with precisely one ideal obtained in \cite{semt_classgraphalg} leads to stable isomorphism of the $C^*$-algebras in question, and is hence only an exact classification result in the -- very important -- special case where the $C^*$-algebra is stable. 
In this final section, drawing on results above, we extend the result to all non-unital graph $C^*$-algebras with one ideal.

The remaining case of exact classification of  a unital graph $C^*$-algebra requires completely different methods and is explained in \cite{ERRstrong}. Note, however, that the $\inin$ case was solved in \cite{segr}, and the $\fifi$ case is trivial. In the remaining cases, one may employ that since the number of vertices must be finite, any simple $AF$ subquotient is either $\mathsf M_n(\CC)$ or $\KKK$.

\begin{lemma}\label{l:full1ideal}
Let $E$ be a graph such that $C^{*} (E)$ has exactly one non-trivial ideal $\mathfrak{I}$.  If $C^{*}(E)$ is not an AF-algebra, then $\mathfrak{I}$ is stable and 
\begin{align*}
\mathfrak{e} : 0 \to \mathfrak{I} \to C^{*} (E) \to C^{*} (E)/ \mathfrak{I} \to 0
\end{align*}
is a full extension.
\end{lemma}

\begin{proof}
By Proposition~6.4 of \cite{semt_classgraphalg}, $\mathfrak{I}$ is stable.  Since $\mathfrak{I}$ is stably isomorphic to a simple graph $C^{*}$-algebra, $\mathfrak{I}$ is either purely infinite or an AF-algebra.  If $\mathfrak{I}$ is purely infinite, then $\mathfrak{e}$ is a full extension since $\mathfrak{e}$ is an essential extension and $\corona{ \mathfrak{I} }$ is simple.  

Suppose $\mathfrak{I}$ is an AF-algebra.  By Proposition~3.10 of \cite{semt_classgraphalg}, 
\begin{align*}
0 \to \mathfrak{I} \otimes \K \to C^{*} (E) \otimes \K \to  \mathfrak{I}  \otimes \K \to 0
\end{align*} 
is a full extension.  By Corollary~\ref{c:fullstable}, $\mathfrak{e}$ is a full extension.
\end{proof}
 
\begin{defin}
Let $\mathfrak{A}$ be a $C^{*}$-algebra. The scale of $K_{0} ( \mathfrak{A} )$ is 
\begin{align*}
\Sigma \mathfrak{A} = \setof{ x \in K_{0} ( \mathfrak{A} ) }{ \text{$x = [ p ]$ for some projection $p$ in $\mathfrak{A}$}}
\end{align*}

For $C^{*}$-algebras $\mathfrak{A}$ and $\mathfrak{B}$, we say that an isomorphism $\ftn{ \alpha }{ K_{0} ( \mathfrak{A} ) }{ K_{0} ( \mathfrak{B} ) }$ is \emph{scale preserving} if either
\begin{itemize}
\item[(a)] $\mathfrak{A}$ and $\mathfrak{B}$ are unital $C^{*}$-algebras and $\alpha ( [ 1_{ \mathfrak{A} } ] ) = [ 1_{ \mathfrak{B} } ]$ or

\item[(b)] $\mathfrak{A}$ and $\mathfrak{B}$ are non-unital $C^{*}$-algebras and $\alpha$ is an isomorphism from $\Sigma \mathfrak{A}$ to $\Sigma \mathfrak{B}$.
\end{itemize}
\end{defin}

\begin{theor}
Let $E_{1}$ and $E_{2}$ be graphs and suppose $C^{*} ( E_{i} )$ is a non-unital $C^{*}$-algebra such that $C^{*} ( E_{i} )$ has exactly one non-trivial ideal $\mathfrak{I}_{i}$.  Then $C^{*} ( E_{1} ) \cong C^{*} ( E_{2} )$ if and only if there exists an isomorphism $\ftn{ ( \beta_{*}, \eta_{*} , \alpha_{*} ) }{ \ksix ( C^{*} ( E_{1} ) ; \mathfrak{I}_{1} ) }{ \ksix ( C^{*} ( E_{2} ) ; \mathfrak{I}_{2} ) }$ such that $\alpha_{0}$, $\beta_{0}$, and $\eta_{0}$ are positive isomorphisms, and  $\eta_0$ and $\alpha_0$ are scale preserving.\footnote{This theorem is not correct as stated.  See arXiv:1505.05951 for more details.}

In the case $\fifi$ all but the scaled ordered groups $K_0(C^*(E_i)))$ may be deleted from the invariant, leaving it complete. In the remaining cases, the orders and scales of $K_0(C^*(E_i))$ may be deleted from the invariant, leaving it complete. 
\end{theor}

\begin{proof}
The ``only if'' direction is clear.  We now prove the ``if'' direction.  If $C^{*} ( E_{1} )$ is an AF-algebra, then $C^{*} ( E_{2} )$ is an AF-algebra since $\alpha_{0}$ and $\beta_{0}$ are positive isomorphisms.  Thus the result follows from Elliott's classification \cite{af}.  

In the case when $C^{*} ( E_{1} )$ is not an AF-algebra, let $\mathfrak{e}_{i}$ be the extension 
\begin{align*}
0 \to \mathfrak{I}_{i} \to C^{*} ( E_{i} ) \to C^{*} ( E_{i} ) / \mathfrak{I}_{i} \to 0.
\end{align*}
By  Lemma~\ref{l:full1ideal}  $\mathfrak{e}_{i}$ is a full extension and $\mathfrak{I}_{i}$ is stable.
We complete the proof by following  the approach in Theorem~3.8 of \cite{ERRshift} adding the technology of \cite{segrer:ccfis}.

Both $\alpha_0$ and $\beta_0$ are induced by $*$-isomorphisms by either \cite{af} or \cite{ekncp:eeccao}. Conjugating by these as in \cite{extpurelyinf} and appealing to Theorem~2.3 of \cite{ERRshift}  we get that there exist isomorphisms $\ftn{ \phi_{ 2 } }{ \mathfrak{I}_{1} }{ \mathfrak{I}_{2} }$ and $\ftn{ \phi_{1} }{ C^{*} ( E_{1} ) / \mathfrak{I}_{1} }{ C^{*} ( E_{2} ) / \mathfrak{I}_{2} }$ such that 
\begin{align*}
\kk ( \phi_{1 } ) \times [ \tau_{ \mathfrak{e}_{2} }  ] = [  \tau_{ \mathfrak{e}_{1} } ] \times \kk ( \phi_{2} )
\end{align*}
in $\kk^{1} ( C^{*} ( E_{1} ) / \mathfrak{I}_{1} , \mathfrak{I}_{2} )$.  Hence, by Lemma~4.5 of \cite{segrer:ccfis}, $C^{*} ( E_{1} ) \cong C^{*} ( E_{2} )$.
\end{proof}

As above, it is possible to describe the scale explicitly except in the $\fifi$ case. We will not do so here. Note that by the Zhang dichotomy (\cite{sz:ppisc}) we do not need to concern ourselves with the scale in the $\fiin$ and $\inin$ cases unless the quotient is unital.

We end this section by comparing stable and exact isomorphism for the class of $C^*$-algebras $C^*(G[m,(n_i)])$ with $1<m<\infty$.
Note that these $C^*$-algebras are not stable, as they have unital quotients $\mathcal O_m$. This observation also implies that for 
$C^*(G[m,(n_i)])$ to be stably isomorphic $C^*(G[m',(n'_i)])$, it is necessary that $m=m'$. We will abbreviate $\ZZ^\infty=\sum_{i=1}^\infty \ZZ$.

\begin{examp}
For $(n_i),(n_i')\in\ZZ^\infty$, choose $k$ such that $n_i=n_i'=0$ whenever $i>k$ and set
\[
N=\sum_{i=1}^k2^{k-i}n_i\qquad
N'=\sum_{i=1}^{k'}2^{k'-i}n'_i.
 \]
Then $C^*(G[m,(n_i)])\cong C^*(G[m,(n'_i)])$ precisely when for some $\ell,\ell'\geq 0$ we have
\[
2^\ell N\equiv 2^{\ell'} N' \mod m-1.
\]
Then $C^*(G[m,(n_i)]) \otimes\mathbb K\cong C^*(G[m,(n'_i)]) \otimes\mathbb K$ precisely when for some $\ell,\ell'\geq 0$ and some unit $x$ of $\ZZ/(m-1)$, we have
\[
2^\ell N\equiv x2^{\ell'} N' \mod m-1
\]
(which is the same as
$
(N,M)=(N',M)
$
for $M$ the largest odd factor of $m-1$).
\end{examp}

\begin{proof}
By our classification result, we just need to compare $\ksix(C^*(G[m,(n_i)]);\mathfrak I)$ and
$\ksix(C^*(G[m',(n_i')]);\mathfrak I)$ where we may identify the ideals since both are isomorphic to the stabilized UHF algebra. 
Note that the order-preserving automorphisms of the $K_0$-group of this ideal are of the form $1\mapsto 2^k$ with $k\in \ZZ$; considering $\ZZ[\frac{1}{2}]$ as the cokernel of the map in $\ZZ^\infty$ given by the matrix
\[
B-I=
\begin{bmatrix}
-1&   &      \\
2& -1&    \\
&\ddots&\ddots
\end{bmatrix}
\]
all such maps are induced by 
\[
\beta_{\ell,\ell'}(\mathbf{e}_i)=2^\ell\mathbf{e}_{i+\ell'}
\]
for $\ell,\ell'\geq 0$. There is only the trivial unit-preserving automorphism of $K_0(\mathcal O_m)$.

Thus, to decide whether or not the $C^*$-algebras are isomorphic, we need to decide if there is $\ell,\ell'$ and an isomorphism $\xi$
such that
\[
\xymatrix{
{\ZZ^\infty}\ar[r]\ar[d]_{\beta_{\ell,\ell'}}&{\ZZ^\infty\oplus \ZZ}\ar[d]_{\xi}\ar[r]&{\ZZ}\ar@{=}[d]\\
{\ZZ^\infty}\ar[r]&{\ZZ^\infty\oplus \ZZ}\ar[r]&{\ZZ}}
\]
commutes and 
\[
\xi\left(\operatorname{im}\begin{bmatrix}&&n_1\\ &B-I&n_2\\ &&\vdots\\ &&m-1\end{bmatrix}\right)\subseteq
\operatorname{im}\begin{bmatrix}&&n_1'\\ &B-I&n_2'\\ &&\vdots\\ &&m-1\end{bmatrix}.
\]
Employing the fact that $\xi(\mathbf 0,1)=(\mathbf v,1)$ for some vector $\mathbf v\in \ZZ^\infty$, we get the stated condition.

In the case of stable isomorphism, the automorphism on $K_0(\mathcal O_m)$ is not required to be unital and is hence given by some unit of $\ZZ/(m-1)$. The computations follow similarly.
\end{proof}

We note that for many $m$, stable isomorphism in this class in fact is the same as exact isomorphism, even though the $C^*$-algebras are never stable. The smallest $m$ for which the notions differ is 8.

\section{Acknowledgements}

The authors would like to thank Takeshi Katsura, Ping Wong Ng and Mark Tomforde for helpful discussions which lead to the sharpening of many of our results.

\end{document}